\documentclass[11pt]{amsart}
\usepackage[T1]{fontenc}
\usepackage{lmodern}
\usepackage{microtype}
\usepackage{mathtools,amssymb}
\usepackage[margin=1.12in]{geometry}
\usepackage{enumitem}
\usepackage{cite}
\usepackage{xcolor}
\usepackage[colorlinks=true,linkcolor=blue!55!black,citecolor=blue!55!black,urlcolor=blue!55!black]{hyperref}
\hypersetup{pdftitle={Prescribed null singular sets for stationary harmonic maps into smooth compact targets},pdfauthor={}}
\numberwithin{equation}{section}
\newtheorem{theorem}{Theorem}[section]
\newtheorem{proposition}[theorem]{Proposition}
\newtheorem{lemma}[theorem]{Lemma}
\newtheorem{corollary}[theorem]{Corollary}
\theoremstyle{definition}

\theoremstyle{remark}
\newtheorem{remark}[theorem]{Remark}
\newcommand{\R}{\mathbb R}
\newcommand{\C}{\mathbb C}
\newcommand{\T}{\mathbb T}
\newcommand{\sph}{\mathbb S^2}
\newcommand{\dd}{\,\mathrm d}
\newcommand{\ep}{\varepsilon}
\newcommand{\la}{\lambda}
\newcommand{\sing}{\operatorname{sing}}
\newcommand{\Div}{\operatorname{div}}
\newcommand{\tr}{\operatorname{tr}}
\newcommand{\capacity}{\operatorname{Cap}}
\newcommand{\dist}{\operatorname{dist}}
\newcommand{\supp}{\operatorname{supp}}
\newcommand{\cH}{\mathcal H}
\newcommand{\cL}{\mathcal L}
\newcommand{\cR}{\mathcal R}
\newcommand{\cF}{\mathcal F}
\newcommand{\cN}{\mathcal N}
\newcommand{\Np}{p_+}
\newcommand{\Sp}{p_-}
\newcommand{\Sph}{\mathbb{S}}
\newcommand{\Sing}{\operatorname{Sing}}

\newcommand{\vol}{\operatorname{vol}}

\newcommand{\norm}[1]{\lVert#1\rVert}
\newcommand{\ip}[2]{\langle#1,#2\rangle}
\newcommand{\Gprod}{G_{\mathrm{prod}}}
\newcommand{\doi}[1]{\href{https://doi.org/#1}{\nolinkurl{doi:#1}}}
\setlist[enumerate]{label=\textup{(\roman*)},leftmargin=2em}
\allowdisplaybreaks[1]
\title[Prescribed singular sets]{Prescribed singular sets for stationary harmonic maps}


\begin{document}
\begin{abstract}
A central conjecture in the regularity theory of harmonic maps asserts that the singular set of a stationary harmonic map has codimension at least three. In this paper, we construct counterexamples to this folklore conjecture.
We construct stationary harmonic maps whose singular sets $S$ satisfies $\dim S=n-2$ and $\cH^{n-2}(S)=0$ . The target is a smooth compact manifold diffeomorphic to $\mathbb{S}^2 \times \mathbb{T}^3$, with metric arbitrarily close in $C^\infty$ to the standard product metric.
\end{abstract}

\author{Wenshuai Jiang, Chang Li and Wenyou Yu}

\address[Wenshuai Jiang]{School of Mathematical Sciences, Zhejiang University, Hangzhou 310058, China}

\address[Chang Li]{School of Mathematics, Renmin University of China, Beijing 100872, China}

\address[Wenyou Yu]{School of Mathematical Sciences, Zhejiang University, Hangzhou 310058, China}

 \email{wsjiang@zju.edu.cn}
 \email{chang\_li@ruc.edu.cn}
 \email{wenyouyu1202@zju.edu.cn}

\maketitle

\section{Introduction}

Harmonic maps form a fundamental class of nonlinear elliptic systems in differential geometry and the calculus of variations.  Given Riemannian manifolds $(M^n,g)$ and $(N,h)$, a map $u:M\to N$ is harmonic if it is a critical point of the Dirichlet energy
\[
E(u;\Omega)=\frac12\int_\Omega |du|^2\,d\vol_g .
\]
Their systematic study began with the foundational work of Eells and Sampson~\cite{EellsSampson1964}; the work of Sacks and Uhlenbeck~\cite{SacksUhlenbeck1981} subsequently highlighted the fundamental role of harmonic two-spheres and energy concentration in geometric variational problems.  Standard references include~\cite{EellsLemaire1978,EellsLemaire1988,Helein2002,Jost1991,LinWang2008,Simon1996}.

A fundamental issue in harmonic map theory is regularity.  In dimension two, weakly harmonic maps are smooth by H\'elein's regularity theory~\cite{Helein1990,Helein1991}.  In higher dimensions, however, Rivi\`ere~\cite{Riviere1995} constructed finite-energy weakly harmonic maps that are discontinuous everywhere.  Additional variational structure is therefore essential.  A weakly harmonic map is called \emph{stationary} if it is also critical under compactly supported deformations of the domain.  Every energy-minimizing harmonic map is stationary, but the converse is false.

For energy-minimizing harmonic maps, Schoen and Uhlenbeck~\cite{SU1982} proved the sharp estimate
\begin{equation}\label{eq:minimizing}
\dim_{\cH}\Sing(u)\le n-3.
\end{equation}
Their argument combines small-energy regularity, strong compactness, and dimension reduction.  Simon~\cite{Simon1995} proved rectifiability of the singular set for real-analytic targets.  Naber and Valtorta~\cite{NV2017} later established $(n-3)$-rectifiability and locally finite $(n-3)$-dimensional Hausdorff measure for general smooth targets, together with sharp quantitative Minkowski estimates.  Sharp examples and further structural results for minimizing harmonic maps can be found in~\cite{BCL1986,JagerKaul1983,HardtLin1990,Luckhaus1988,Okayasu1994}; related developments for minimizing $p$-harmonic maps appear in~\cite{HardtLin1987}.

The stationary theory is substantially more delicate.  Stationarity yields a monotonicity formula, but does not by itself provide the strong compactness that holds for minimizers. Evans~\cite{Evans1991}, for sphere-valued maps, and Bethuel~\cite{Bethuel1993}, for general smooth targets, proved small-energy regularity and showed that
\begin{equation}\label{eq:classical}
\cH^{n-2}\bigl(\Sing(u)\bigr)=0.
\end{equation}
Thus the singular set has codimension at least two in the measure-theoretic sense, but \eqref{eq:classical} neither implies a strict Hausdorff-dimension bound nor reaches the minimizing threshold \eqref{eq:minimizing}.  Alternative analytic approaches to partial regularity were developed in~\cite{ChangWangYang1999,RiviereStruwe2008}.

Lin's blow-up theory~\cite{Lin1999} identified the principal obstruction.  For a sequence of stationary harmonic maps with uniformly bounded energy, one may have
\[
u_i\rightharpoonup u\quad\text{in }W^{1,2},
\qquad
|du_i|^2d\vol_g\rightharpoonup |du|^2d\vol_g+\nu,
\]
where the defect measure is of the form
\[
\nu=\vartheta\,\cH^{n-2}\!\restriction\Sigma,
\qquad
\Sigma\ \text{countably $(n-2)$-rectifiable}.
\]
Lin related this codimension-two concentration to nonconstant harmonic two-spheres in the target.  Their absence rules out the corresponding concentration mechanism and leads to strong compactness; under additional hypotheses, including analyticity of the target in Lin's sharper singular-set theorem, one obtains improved dimension and rectifiability estimates.  Li and Tian~\cite{LiTian1998} obtained a blow-up formula in this setting, while Ding, Li, and Li~\cite{DingLiLi2003} showed that a weak limit of stationary harmonic maps need not remain stationary.  These phenomena demonstrate that the weak limit alone need not retain all the geometric information carried by the sequence.

The lost energy is now understood with remarkable precision.  Lin and Rivi\`ere~\cite{LinRiviere2002} proved energy quantization for sphere targets. A major breakthrough was achieved by Naber and Valtorta~\cite{NVEnergy2024}, who established the general energy identity using quantitative stratification: at $\cH^{n-2}$-almost every point of the defect set,
\begin{equation}\label{eq:energy-identity}
\vartheta(x)=\sum_{\alpha=1}^{L(x)}E(\omega_{x,\alpha}),
\qquad
\omega_{x,\alpha}:\Sph^2\longrightarrow N
\end{equation}
for finitely many nonconstant harmonic two-spheres.  Thus all the defect energy is accounted for by the bubbles.

Quantitative stratification gives a complementary description.  Cheeger and Naber~\cite{CheegerNaber2013} developed quantitative estimates for strata defined by the approximate symmetries of tangent maps, and Naber and Valtorta~\cite{NV2017} proved sharp estimates and rectifiability results for these strata.  This theory provides much finer information than a single dimension bound, but information carried by tangent maps must still be distinguished from information carried by limiting energy measures.  In particular, singular points may have constant weak tangent maps, as already occurs in the examples of Hardt, Lin, and Poon~\cite{HLP1992}; see also~\cite[Example~1.2]{Lin1999}.  Thus a direct dimension-reduction argument based only on nonconstant tangent maps cannot capture all possible singular behavior, see also ~\cite{CJN2026}. 

Stronger conclusions are available under additional stability or target assumptions.  Results of Hong~\cite{Hong1999}, Hong and Wang~\cite{HongWang1999}, Hsu~\cite{Hsu2005}, Hsu and Li~\cite{HsuLi2008}, and Lin and Wang~\cite{LinWang2006} show that excluding suitable stable harmonic two-spheres or homogeneous tangent maps can force higher codimension or restore compactness, see also recent progresses \cite{LiX, Kra}.  Closely related results on stability and the Morse index of harmonic two-spheres include~\cite{Xin1980,ElSoufi1995,Karpukhin2021}, while particular singular models and instability phenomena are studied in~\cite{Nakajima2006,Nakajima2009}.  Finite-Morse-index stationary harmonic maps and their regularity play an important role in the recent existence theory of Karpukhin and Stern~\cite{KarpukhinStern2024}.  These results underscore the regularizing effect of stability, but they do not resolve the unrestricted stationary problem.

For general stationary harmonic maps, the central conjecture is
\begin{equation}\label{eq:conjecture}
\dim_{\cH}\Sing(u)\le n-3.
\end{equation}
This is the \emph{codimension-three conjecture}; it is also formulated as Conjecture~1.8 in the  monograph of Chow, Jiang, and Naber~\cite{CJN2026}, see also \cite{NVEnergy2024}.
Despite these substantial developments, there has been little direct progress toward the codimension-three conjecture itself. Actually, we show that it fails when the smooth compact target metric is unrestricted. In fact, any compact null subset of an axis can occur as the full singular set. The metric is part of the construction; no minimizing or stability property is asserted, and the result does not concern a fixed round target.

\subsection{Main results}

Fix $L=10$, put $\T=\R/L\mathbb Z$ and $X=\T^3$, and give $X$ its flat metric. Write its coordinates locally as $x=(x_1,x_2,z)$. For a nonempty compact set $E\subset\T$ of one-dimensional Lebesgue measure zero, define
\[
S_E=\{(0,0,z):z\in E\}\subset X.
\]
The target $\sph\times X$ carries the reference product metric $\Gprod=g_{\sph}+g_X$, with $g_{\sph}$ the unit round metric.

\begin{theorem}\label{thm:main}
For every nonempty compact null set $E\subset\T$, there are smooth Riemannian metrics $G_\delta$ on $\sph\times X$, defined for all sufficiently small $\delta>0$, and maps
\[
U_\delta\in W^{1,2}(X,(\sph\times X,G_\delta))
\]
such that:
\begin{enumerate}
\item $U_\delta$ is weakly harmonic and stationary, and is smooth on $X\setminus S_E$;
\item every point of $S_E$ is an essential discontinuity of $U_\delta$, so $\sing U_\delta=S_E$;
\item $G_\delta\to\Gprod$ in $C^\infty$ as $\delta\downarrow0$;
\item the energies $E_{G_\delta}(U_\delta;X)$ are bounded independently of sufficiently small $\delta$.
\end{enumerate}
Each example uses one fixed smooth positive metric; no limiting metric is needed to obtain its singular set.
\end{theorem}

\begin{corollary}\label{cor:sharp}
For every $n\ge3$ and every $R>0$, there exist a smooth compact five-dimensional Riemannian manifold $(N,G)$ and a finite-energy stationary weakly harmonic map $u:B_R^n\to N$ such that
\[
\dim_H\sing u=n-2,\qquad \cH^{n-2}(\sing u)=0.
\]
In particular, the bound $\dim_H\sing u\le n-3$ fails for unrestricted smooth compact targets. The metric $G$ can be chosen arbitrarily close in $C^\infty$ to the round--flat product metric on $\sph\times\T^3$.
\end{corollary}

The conclusion is compatible with \eqref{eq:classical}: the Hausdorff measure remains zero even when the Hausdorff dimension is $n-2$. The freedom to prescribe $E$ includes finite sets and null Cantor sets of Hausdorff dimension one. Thus the examples are not restricted to isolated singularities or to a fixed model of a fractal set.

\vspace{.5cm}
% The background and main statements occupy the first three pages.
\subsection{Outline of the construction}\label{sec:outline}
The proof separates the construction of a sphere-valued map from the choice of the target metric. 
The intermediate map $q:X\to\sph$ need not be harmonic everywhere. 
Instead, its tension and the divergence of its stress tensor will be smooth and small, although $q$ has the prescribed essential singularities. 
An explicit metric correction then makes the graph $U=(q,\mathrm{id})$ weakly harmonic and stationary.

\smallskip
\noindent\emph{The metric correction.}
Theorem~\ref{prop:repair} gives the general reduction. Suppose that $q$ is smooth away from a set of zero $2$-capacity and that its tension $e$ extends smoothly to $X$. The additional requirement is the full distributional identity
\[
\Div T_q=\cF,\qquad
T_q=|Dq|^2I-2(Dq)^TDq,\qquad
\cF_j=-2e\cdot\partial_jq.
\]
Solve $\Delta V=\cF/2$, put $H_{ij}=\partial_iV_j+\partial_jV_i$, and define the metric
\[
G=g_{\sph}+\left(I+H(y)+\tfrac23(e(y)\cdot p)I\right)_{ij}
\dd y^i\dd y^j.
\]
The term linear in $p$ cancels the spherical Euler equation. The Poisson equation cancels both the torus Euler equation and the stress divergence. The verification uses arbitrary target and domain variations, not only variations preserving the symmetry of the construction.

\smallskip
\noindent\emph{The spherical map.}
On each gap $(a,b)$ of $\T\setminus E$, choose
\[
\la_\delta(z)=\delta\exp\!\left(-\frac1{(z-a)(b-z)}\right).
\]
The transverse bubble $F(r,z)=2\arctan(r/\la_\delta(z))$ has the required behavior at the axis, but is not harmonic when its scale varies with $z$. We solve the equivariant harmonic-map equation with the same analytic axis data. The main point is to continue the solution to a radius much larger than $\la_\delta$. A radial Volterra inverse gains two powers of $r$ and, after the first logarithmic resonance, gives a quadratic denominator in the iteration order. Weighted Cauchy estimates balance this denominator against the loss from two axial derivatives. The resulting convergent construction has canonical coefficients and agrees on overlapping axial disks.

We cut the exact solution off to the south pole at the radius
\[
\rho(z)=\frac{\kappa}{1+|\log\delta|}
\left(1+\frac1{(z-a)(b-z)}\right)^{-4}.
\]
The separation $\la_\delta\ll\rho$ makes the cutoff error small. Exponential flatness makes every derivative of the tension vanish at $S_E$, while the map itself oscillates between the two poles arbitrarily close to each point of $S_E$. Proposition~\ref{prop:construction} gives the resulting map.

It is worth noting that the second step only relies on the fact that closed set $E\subset \mathbb{T}$ contains no interior point while the nullity of $E$ is crucial for the third step.

\smallskip
\noindent\emph{The distributional verification.}
Proposition~\ref{prop:verification} verifies the hypotheses of Theorem~\ref{prop:repair}. The energy estimate follows from the bubble bounds, and the weak equation extends by a capacity argument. Stationarity requires a separate calculation: extending the classical stress divergence smoothly across $S_E$ does not exclude a distribution supported on $S_E$. We compare the stress with that of the explicit varying bubble and compute its flux through shrinking tubes. The axial flux is a $z$-derivative. After integration by parts, its limit vanishes because $E$ has Lebesgue measure zero. This gives the stress identity on the entire domain.

We then apply the metric correction and take products in the unused domain variables. A compact null Cantor set of Hausdorff dimension one gives Corollary~\ref{cor:sharp}. 

The spherical fibers of the corrected target remain totally geodesic round two-spheres. Hence assumptions excluding harmonic two-spheres do not apply. The metric is smooth, but is not required to be real analytic; no assertion is made for real-analytic target metrics.

The paper is orgnaized as follows: Section~\ref{sec:repair} proves Theorem~\ref{prop:repair}. Section~\ref{sec:construction} constructs $q_\delta$ by analytic continuation and cutoff. Section~\ref{sec:verification} verifies its energy, tension, and stress properties. Section~\ref{sec:main-proof} completes the main proofs.

\vspace{.25cm}
\textbf{AI usage.} The authors made substantial use of ChatGPT 6.0. Motivated by Fanghua Lin’s results for targets without nonconstant harmonic two-spheres, we directed the AI to explore target manifolds with an \(\mathbb S^2\) factor. The AI initially generated a low-dimensional prototype, but its presentation was unclear and several key arguments were incomplete. Through extensive author-guided interaction, this prototype was developed into a stronger counterexample. The authors then verified the arguments, filled in the missing details, and reorganized the material into the present exposition. The authors have carefully reviewed the final manuscript and take full responsibility for its mathematical content and proofs.

\section{Reduction to a sphere-valued map}\label{sec:repair}
We reduce the construction of stationary harmonic maps to a problem for a sphere-valued map with prescribed singularities. The inputs are a smooth tension field and a smooth stress divergence, the latter taken on the whole domain. The following theorem gives an explicit correction of the target metric.

Throughout this section, $Dq$ has columns $\partial_iq$. For a sphere-valued map define
\begin{equation}\label{eq:stress-q}
T_q=|Dq|^2I-2(Dq)^TDq.
\end{equation}

\begin{theorem}\label{prop:repair}
Let $S\subset X=\mathbb{T}^3$ be closed, and suppose
\[
\capacity_2(S)=0,\qquad
q\in W^{1,2}(X,\sph)\cap C^\infty(X\setminus S,\sph).
\]
Assume there is $e\in C^\infty(X,\R^3)$ such that
\[
\Delta q+|Dq|^2q=e\quad\text{on }X\setminus S,
\]
and that
\begin{equation}\label{eq:repair-assumptions}
\Div T_q=\cF,\qquad \cF_j=-2e\cdot\partial_jq
\end{equation}
on all of $X$ in distributions, where $\cF$ is smooth. If $\norm e_{C^0}+\norm\cF_{C^0}$ is sufficiently small, there is a smooth positive metric $G$ on $\sph\times X$ for which $U(x)=(q(x),x)$ is stationary and weakly harmonic. For every $m\ge0$,
\begin{equation}\label{eq:metric-Cm}
\norm{G-\Gprod}_{C^m}
\le C_m\bigl(\norm e_{C^m}+\norm\cF_{C^m}\bigr).
\end{equation}
\end{theorem}

\begin{remark}
We point out that the weaker assumption $\cH^1(S)<\infty$ already implies $\capacity_2(S)=0$, which is sufficient for the argument here. The stronger assumption $\cH^1(S)=0$ is only used later in Proposition \ref{prop:stress} where we verify that the constructed map $q$ satisfies the hypotheses of Theorem \ref{prop:repair}.
\end{remark}

\subsection{Weak equations and extension across capacity-zero sets}\label{sec:capacity}
If $q$ is smooth and $e=\Delta q+|Dq|^2q$, then
\begin{equation}\label{eq:stress-identity-smooth}
(\Div T_q)_j=-2e\cdot\partial_jq.
\end{equation}
For a map into a general target the corresponding tensor is
\[
(T_u)_{ij}=|Du|_G^2\delta_{ij}-2G(u)(\partial_i u,\partial_j u).
\]
For a finite-energy weakly harmonic map,
\[
u\text{ is stationary}
\ \Longleftrightarrow\ \Div T_u=0\text{ in }\mathcal D'
\ \Longleftrightarrow\ \int T_u:DY=0\quad\forall\,Y\in C_c^\infty.
\]
The equivalence follows by changing variables in the energy under a smooth domain flow. Tensor $C^m$ norms are taken with respect to the fixed background metrics. Constants may increase between occurrences; uniformity in the bubble scale is stated explicitly in the analytic estimates below.

We record the capacity argument used below.
\begin{lemma}\label{lem:capacity}
Let $\Omega$ be a smooth domain or a smooth Riemannian manifold, and let $S\subset\Omega$ be relatively closed with locally zero Sobolev $2$-capacity. 

\begin{samepage}
\textup{(i)} For a map $v$ defined on $\Omega\setminus S$, denote $\widetilde v$ for an extension with arbitrary bounded values on $S$. If
\[
v\in C^\infty(\Omega\setminus S,\R^m)\cap L^\infty(\Omega\setminus S),
\qquad \int_{\Omega\setminus S}|Dv|^2\dd\mathrm{vol}<\infty,
\]
then
\[
\begin{gathered}
\widetilde v\in W^{1,2}_{\mathrm{loc}}(\Omega,\R^m),\qquad
D\widetilde v=1_{\Omega\setminus S}Dv\quad\text{a.e.},\\
\int_\Omega|D\widetilde v|^2\dd\mathrm{vol}
=\int_{\Omega\setminus S}|Dv|^2\dd\mathrm{vol}.
\end{gathered}
\]
Moreover, $\widetilde v\in W^{1,2}(\Omega)$ if and only if $\widetilde v\in L^2(\Omega).$
In particular, $\operatorname{vol}(\Omega)<\infty$ implies $\widetilde v\in W^{1,2}(\Omega)$.
\end{samepage}

\textup{(ii)} Let $(N,G)$ be smooth and compact. If
\[
u\in W^{1,2}_{\mathrm{loc}}(\Omega\setminus S,N),\qquad
E_G(u;\Omega\setminus S)<\infty,\qquad
\tau_G(u)=0\quad\text{in }\mathcal D'(\Omega\setminus S),
\]
then every $N$-valued extension $\widetilde u$ satisfies
\[
\begin{gathered}
\widetilde u\in W^{1,2}_{\mathrm{loc}}(\Omega,N),\qquad
\tau_G(\widetilde u)=0\quad\text{in }\mathcal D'(\Omega),\\
E_G(\widetilde u;\Omega)=E_G(u;\Omega\setminus S).
\end{gathered}
\]
Here $\tau_G=0$ denotes the full weak target-variation equation.
Moreover, $\operatorname{vol}(\Omega)<\infty $ implies $
\widetilde u\in W^{1,2}(\Omega,N)$. 
No smoothness of $u$ on $\Omega\setminus S$ is assumed.
\end{lemma}
\begin{proof}
We first specify the capacity convention. For every $K\Subset\Omega$, choose an open $W\Subset\Omega$ containing $K$ such that
\[
\inf\left\{
\int_W(|a|^2+|Da|^2)\dd\mathrm{vol}:
 a\in C_c^\infty(W),\quad a\ge1\ \text{near }S\cap K
\right\}=0.
\]
On a compact manifold we may take $W=\Omega$. Choose a smooth truncation $T$ with
\[
T|_{(-\infty,0]}=0,\qquad T|_{[1,\infty)}=1,\qquad
0\le T\le1,\qquad |T(t)|\le C|t|,\quad |T'(t)|\le C.
\]
For admissible $a_k$ with sufficiently small $W^{1,2}$ norm, $\eta_k=T(a_k)$ satisfies
\[
\eta_k\in C_c^\infty(W),\quad
0\le\eta_k\le1,\quad \eta_k=1\ \text{near }S\cap K,\quad
\|\eta_k\|_{L^2}+\|D\eta_k\|_{L^2}\le2^{-k}.
\]
Tonelli's theorem gives
\[
\int_W\sum_k\eta_k^2\dd\mathrm{vol}
=\sum_k\|\eta_k\|_{L^2}^2<\infty
\quad\Longrightarrow\quad \eta_k\longrightarrow0\quad\text{a.e.}
\]
Thus $\zeta_k=1-\eta_k$ satisfies
\[
0\le\zeta_k\le1,\quad \zeta_k=0\ \text{near }S\cap K,\quad
\zeta_k\to1\ \text{a.e.},\quad \|D\zeta_k\|_{L^2}\to0.
\]
Moreover,
\[
\operatorname{vol}(S\cap K)\le\|\eta_k\|_{L^2}^2\to0,
\qquad \operatorname{vol}(S)=0
\]
by exhaustion. The cutoffs depend on the compact support of each test field; no single sequence near all of a noncompact $S$ is required.

We prove \textup{(i)} in Euclidean coordinates. On relatively compact Riemannian charts the coordinate and metric norms are uniformly equivalent, so localization gives the same conclusion. Put $g=1_{\Omega\setminus S}Dv\in L^2(\Omega)$. Given $\psi\in C_c^\infty(\Omega)$, use the cutoffs for $K=\supp\psi$. Since $\zeta_k\psi\in C_c^\infty(\Omega\setminus S)$,
\[
\int_\Omega\widetilde v^\alpha\zeta_k\partial_i\psi
=-\int_\Omega g_i^\alpha\zeta_k\psi
 -\int_\Omega\widetilde v^\alpha\psi\partial_i\zeta_k.
\]
The last term satisfies
\[
\left|\int_\Omega\widetilde v^\alpha\psi\partial_i\zeta_k\right|
\le\|\widetilde v\|_{L^\infty}\|\psi\|_{L^2}
\|D\zeta_k\|_{L^2}\longrightarrow0.
\]
The first two terms are dominated by $\|\widetilde v\|_{L^\infty}|D\psi|$ and $|g||\psi|$, which belong to $L^1$. Hence
\[
\int_\Omega\widetilde v^\alpha\partial_i\psi
=-\int_\Omega g_i^\alpha\psi,
\qquad D\widetilde v=g\quad\text{in }\mathcal D'(\Omega).
\]
Boundedness gives $\widetilde v\in L^2_{\mathrm{loc}}$, and finite volume gives $\widetilde v\in L^2$. This proves \textup{(i)}. The global $L^2$ condition is necessary on an infinite-volume domain: on $\R^n$,
\[
v\equiv1,\qquad Dv=0,\qquad
\int_{\R^n}|Dv|^2=0,\qquad v\notin L^2(\R^n).
\]

The same calculation applies to $v\in W^{1,2}_{\mathrm{loc}}(\Omega\setminus S)\cap L^\infty$ with $Dv\in L^2$: integration by parts is then the weak derivative identity. Fix a smooth embedding $\iota:N\hookrightarrow\R^M$. Compactness gives
\[
\|\iota\circ u\|_{L^\infty}<\infty,\qquad
c|Du|_G\le|D(\iota\circ u)|\le C|Du|_G.
\]
Applying the Sobolev argument to $\iota\circ u$ extends $u$ across $S$. Since $S$ is null,
\[
\widetilde u\in N\quad\text{a.e.},\qquad
E_G(\widetilde u;\Omega)=E_G(u;\Omega\setminus S).
\]
We suppress the tilde in the rest of the proof.

It remains to verify the full weak harmonic equation. Let $\Xi(x,p)\in T_pN$ be smooth, with $\operatorname{pr}_\Omega(\supp\Xi)\subset K\Subset\Omega$. In Euclidean coordinates set
\[
A_\Xi(x)=\sum_i
\ip{\partial_i u}{\partial_{x_i}\Xi(x,u)
 +\nabla^{N,p}_{\partial_i u}\Xi(x,u)}_G.
\]
The derivative $\partial_{x_i}$ holds $p$ fixed; $\nabla^{N,p}$ differentiates in $N$. On a Riemannian domain take the corresponding trace in a local orthonormal frame. For
\[
u_t(x)=\exp_{u(x)}\bigl(t\Xi(x,u(x))\bigr)
\]
we have
\[
|A_\Xi|\le C_\Xi(|Du|_G+|Du|_G^2)1_K\in L^1(\Omega).
\]
Smoothness of $(x,p,t)\mapsto\exp_p(t\Xi(x,p))$ and compactness of $K\times N$ give, for $0<|t|<t_0$,
\[
\left|\frac{|Du_t|_G^2-|Du|_G^2}{t}\right|
\le C_\Xi(1+|Du|_G^2)1_K\in L^1(\Omega).
\]
Differentiation under the integral is therefore justified, and
\[
\left.\frac{d}{dt}\right|_{t=0}E_G(u_t;\Omega)
=2\int_\Omega A_\Xi\dd\mathrm{vol}.
\]
Use the cutoffs associated with $K$ and put $\Xi_k(x,p)=\zeta_k(x)\Xi(x,p)$. Its domain support avoids $S$. Weak harmonicity on $\Omega\setminus S$ and the product rule give
\[
0=\int_\Omega\zeta_k A_\Xi\dd\mathrm{vol}
+\sum_i\int_\Omega(\partial_i\zeta_k)
\ip{\partial_i u}{\Xi(x,u)}_G\dd\mathrm{vol}.
\]
The second term tends to zero by
\[
\left|\sum_i\int_\Omega(\partial_i\zeta_k)
\ip{\partial_i u}{\Xi(x,u)}_G\dd\mathrm{vol}\right|
\le\sup_{K\times N}|\Xi|_G\,
\|Du\|_{L^2(K)}\|D\zeta_k\|_{L^2(K)}\to0.
\]
Dominated convergence in the first term yields
\[
\int_\Omega A_\Xi\dd\mathrm{vol}=0\quad\text{for every such }\Xi,
\qquad \tau_G(u)=0\quad\text{in }\mathcal D'(\Omega).
\]
This proves \textup{(ii)}.
\end{proof}

In dimension three, $\cH^1(S)<\infty$ implies locally zero Sobolev $2$-capacity.  Compactness then gives zero Sobolev $2$-capacity for $S_E$ on $X$.

\subsection{Proof of Theorem~\ref{prop:repair}}
The Poisson equation supplies a tensor that cancels the torus Euler expression and the stress divergence. A term linear in the sphere coordinate corrects the spherical equation and has zero pullback along the graph.

\begin{proof}
The distributional divergence in \eqref{eq:repair-assumptions} has zero mean. Solve on $X$
\begin{equation}\label{eq:poisson-H}
\Delta V=\cF/2,\qquad \int_XV=0,\qquad
H_{ij}=\partial_iV_j+\partial_jV_i.
\end{equation}
The first derivative of the periodic Green function is integrable in dimension three. Commuting derivatives with convolution therefore gives
\begin{equation}\label{eq:H-Cm}
\norm H_{C^m}\le C_m\norm\cF_{C^m}.
\end{equation}
The elementary identities
\begin{align}
\Div H-\tfrac12d\tr H&=\Delta V,\label{eq:H-Euler}\\
\Div\bigl((\tr H)I-2H\bigr)&=-2\Delta V=-\cF\label{eq:H-stress}
\end{align}
will cancel the two residuals.

For target coordinates $(p,y)\in\sph\times X$, put
\begin{equation}\label{eq:metric}
G=g_{\sph}+K_{ij}(p,y)\dd y^i\dd y^j,
\qquad
K(p,y)=I+H(y)+\frac23(e(y)\cdot p)I.
\end{equation}
This is a globally defined smooth tensor. It is positive if
\[
\norm H_{C^0}+\tfrac23\norm e_{C^0}<\tfrac12.
\]
Estimate \eqref{eq:metric-Cm} follows from \eqref{eq:H-Cm}.

We verify the full Euler equations on $X\setminus S$. Since $|q|=1$, the classical tension satisfies $e\cdot q=0$. Along $U=(q,\mathrm{id})$, the sphere gradient of $\tr K$ is $2e$. The sphere Euler equation is thus
\[
\tau_{\sph}q-\tfrac12\nabla_{\sph}\tr K=e-e=0.
\]
For the $j$th torus coordinate the lowered Euler expression is
\begin{equation}\label{eq:marker-Euler}
\partial_i[K_{ij}(q(x),x)]
-\tfrac12(\partial_{y^j}\tr K)(q(x),x).
\end{equation}
The derivatives in \eqref{eq:marker-Euler} are distinguished by
\[
\partial_i[K_{ij}(q(x),x)]
=(D_pK_{ij})(q(x),x)[\partial_iq]
 +(\partial_{y^i}K_{ij})(q(x),x),
\]
whereas $\partial_{y^j}\tr K$ holds $p$ fixed. Using $e(x)\cdot q(x)=0$, expression \eqref{eq:marker-Euler} equals
\[
\partial_iH_{ij}-\tfrac12\partial_j\tr H
-(\partial_je)\cdot q
=\Delta V_j+e\cdot\partial_jq=0.
\]
These are lowered equations for all target coordinates, so positivity of $G$ makes them equivalent to the full harmonic-map equation. Lemma~\ref{lem:capacity} extends weak harmonicity across $S$.

The last scalar term in \eqref{eq:metric} has zero pullback along $U$. Consequently
\begin{equation}\label{eq:graph-stress}
T_U=T_q+I+(\tr H)I-2H.
\end{equation}
Equations \eqref{eq:repair-assumptions} and \eqref{eq:H-stress} give $\Div T_U=0$ on all of $X$. This proves stationarity under arbitrary domain variations.
\end{proof}

\begin{remark}
The identity \eqref{eq:repair-assumptions} is required on all of $X$. In fact,
\[
\Div T_q=\cF+\mathcal R,\quad\supp\mathcal R\subset S
\quad\Longrightarrow\quad \Div T_U=\mathcal R.
\]
Thus a smooth extension of the classical divergence off $S$ is insufficient unless $\mathcal R=0$. Subsection~\ref{sec:stress} establishes the required identity directly.
\end{remark}

\section{Construction of the spherical map}\label{sec:construction}
In this section, we construct the sphere-valued map required by Theorem~\ref{prop:repair}. First we continue a varying spherical bubble analytically; then we cut it off at a much larger radius. This section gives the map and its classical properties off the prescribed set. Its Sobolev and distributional properties are verified in Section~\ref{sec:verification}. We use the quotient map $p:\R^2\to\T^2$ and the disk $D=B_1^{\R^2}(0)$; since $L=10$, $p|_D$ is injective. The main result in this section is the following

\begin{proposition}\label{prop:construction}
Let $E\subset\T$ be a nonempty compact null set. There are constants $c_h>0$, $0<\kappa\le1/4$, and $0<\delta_0<e^{-1}$ for which the following construction is defined when $0<\delta<\delta_0$. On each lifted gap $(a,b)$ of $\T\setminus E$, put
\[
A(z)=\frac1{(z-a)(b-z)},\quad
\la_\delta(z)=\delta e^{-A(z)},\quad
\rho(z)=\frac{\kappa}{1+|\log\delta|}(1+A(z))^{-4}.
\]
For $0<t\le2$, define
\[
\begin{aligned}
\mathcal C_\delta&=\{(r,z):z\in\T\setminus E,\ 0\le r<2\rho(z)\},\\
\Omega_\delta(t)&=\{(p(\xi),z):\xi\in D,\ z\in\T\setminus E,\ |\xi|<t\rho(z)\}.
\end{aligned}
\]
Then there is a canonical real-analytic solution $f_\delta$ of \eqref{eq:axial} on $\mathcal C_\delta$, with
\[
f_\delta(0,z)=0,\qquad \partial_r f_\delta(0,z)=2/\la_\delta(z)
\quad(z\in\T\setminus E).
\]
Fix $\chi\in C^\infty([0,\infty))$ with $0\le\chi\le1$, $\chi=1$ on $[0,1/2]$, and $\chi=0$ on $[3/4,\infty)$.

For $(p(\xi),z)\in\Omega_\delta(2)$ with $\xi=r(\cos\theta,\sin\theta)\ne0$, define
\begin{equation}\label{eq:q-construction}
\begin{split}
\widetilde f_\delta(r,z)&=\pi+\chi(r/\rho(z))(f_\delta(r,z)-\pi),\\
q_\delta(p(\xi),z)&=(\sin\widetilde f_\delta\cos\theta,
\sin\widetilde f_\delta\sin\theta,\cos\widetilde f_\delta).
\end{split}
\end{equation}
Let $\Np=(0,0,1)$, $\Sp=(0,0,-1)$, and choose any $v_E:S_E\to\sph$.

Extend $q_\delta$ outside  $\Omega_\delta(2)$ by
\[
q_\delta(x)=
\begin{cases}
\Np,&x=(0,0,z),\quad z\in\T\setminus E,\\
\Sp,&x\in X\setminus\bigl(\Omega_\delta(2)\cup S_E\bigr),\\
v_E(x),&x\in S_E.
\end{cases}
\]
Then
\[
\begin{gathered}
q_\delta\in C^\infty(X\setminus S_E,\sph),\qquad
\tau_{\sph}(q_\delta)=0\quad\text{on }\Omega_\delta(1/2),\\
q_\delta=\Sp\quad\text{on }X\setminus
\bigl(\Omega_\delta(3/4)\cup S_E\bigr),
\end{gathered}
\]
and
\[
\overline{\{x\in X:q_\delta(x)\ne\Sp\}}
\subset p\bigl(\overline B_{3\kappa/8}^{\R^2}(0)\bigr)\times\T
\Subset p(D)\times\T.
\]
The estimates of Theorem~\ref{thm:tube} hold uniformly over the gaps, and
$\sup_{z\notin E}\la_\delta(z)/\rho(z)\to0$ as $\delta\downarrow0$.
\end{proposition}

We prove the proposition in two steps. Theorem~\ref{thm:tube} gives estimates uniform in the bubble scale. The flat scale on each complementary interval then permits the local solutions to be assembled and cut off.

\subsection{The equation and the radial inverse}\label{sec:analytic}
For an equivariant sphere map
\[
q_f(r,\theta,z)=(\sin f(r,z)\cos\theta,\sin f(r,z)\sin\theta,\cos f(r,z)),
\]
the harmonic-map equation is
\begin{equation}\label{eq:axial}
f_{rr}+r^{-1}f_r+f_{zz}-\frac{\sin(2f)}{2r^2}=0.
\end{equation}
Let $D_h(z_0)\subset\C$ be a disk with real center. Suppose $\la$ is holomorphic there, real and positive on its real diameter, and
\begin{equation}\label{eq:lambda-comparison}
\ep=\la(z_0)>0,\qquad |\la(z)/\ep-1|\le1/20.
\end{equation}
Define
\begin{equation}\label{eq:bubble-data}
F(r,z)=2\arctan\frac r{\la(z)},\qquad
Z_\la(r,z)=\frac{2\la(z)r}{r^2+\la(z)^2},\qquad
\ell_\ep(r)=1+\log(1+r/\ep).
\end{equation}
The real majorants $Z_\ep$ and $\ell_\ep$ are evaluated only at nonnegative real arguments. For complex radial variables use the domains
\[
\begin{split}
\mathcal D_\ep
&=\{r\in\C:0<|r|,\ |\arg r|<\pi/32\}
  \cup D_{\ep/4}(0),\\
\mathcal D_{\ep,R}&=\mathcal D_\ep\cap D_R(0),\qquad R>0.
\end{split}
\]
These domains contain the axis and avoid $r=\pm i\la(z)$ under \eqref{eq:lambda-comparison}. The branch of $F$ is fixed by $F(0,z)=0$ and continuation along radial segments; equivalently,
\[
F(r,z)=\int_0^r\frac{2\la(z)}{t^2+\la(z)^2}\dd t.
\]
This defines a jointly holomorphic function on $\mathcal D_\ep\times D_h(z_0)$.

\begin{lemma}\label{lem:radial}
Set
\[
\cL_\la=\partial_{rr}+r^{-1}\partial_r-\frac{\cos(2F)}{r^2},
\qquad m_\la=rZ_\la^2.
\]
Define the radial integral operator
\begin{equation}\label{eq:Volterra}
(\cR_\la H)(r,z)=Z_\la(r,z)
\int_0^r\frac{ds}{m_\la(s,z)}
\int_0^s m_\la(t,z)\frac{H(t,z)}{Z_\la(t,z)}\dd t.
\end{equation}
The following estimates apply whenever the defining integrals exist. There is a universal constant $C$ such that, if $|H(t,z)|\le Z_\ep(t)$ for $0<t<r$, then
\begin{equation}\label{eq:first-resonance}
|\cR_\la H(r,z)|\le C Z_\ep(r)r^2\ell_\ep(r).
\end{equation}
For integers $n\ge2$, $q\ge0$, if
$|H(t,z)|\le Z_\ep(t)t^{2n-2}\ell_\ep(t)^q$ for $0<t<r$, then
\begin{equation}\label{eq:radial-n}
|\cR_\la H(r,z)|
\le\frac{C}{n(n-1)}Z_\ep(r)r^{2n}\ell_\ep(r)^q.
\end{equation}

More generally, let $w:D_h(z_0)\to[0,\infty)$ be any finite-valued function and put
\[
\|a\|_w=\sup_{z\in D_h(z_0)}w(z)|a(z)|.
\]
This quantity may be infinite, and no regularity or strict positivity of $w$ is assumed. For every $M\ge0$, the implication
\[
\|H(t,\cdot)\|_w\le MZ_\ep(t)\quad(0<t<r)
\quad\Longrightarrow\quad
\|\cR_\la H(r,\cdot)\|_w
\le CMZ_\ep(r)r^2\ell_\ep(r)
\]
holds. \begin{samepage}
Likewise,
\[
\|H(t,\cdot)\|_w
\le MZ_\ep(t)t^{2n-2}\ell_\ep(t)^q\quad(0<t<r)
\]
implies
\[
\|\cR_\la H(r,\cdot)\|_w
\le\frac{CM}{n(n-1)}Z_\ep(r)r^{2n}\ell_\ep(r)^q.
\]
\end{samepage}
The constant is independent of $w,M,n,q$. The weight is held fixed throughout each radial integration. All these estimates also hold for $r\in\mathcal D_\ep$, with the majorants evaluated at the moduli of the radial variables and the integrals taken along the segment from $0$ to $r$.

For the holomorphic continuation assertion, suppose in addition that $H$ is jointly holomorphic on $\mathcal D_{\ep,R}\times D_h(z_0)$ and has an odd holomorphic germ of order at least $r$ along the axis. Precisely, near every $(0,z_*)$ it has a representation
\[
H(r,z)=r\widehat H(r^2,z)
       =\sum_{j=0}^{\infty}h_j(z)r^{2j+1},
\]
where $\widehat H$ and the coefficients $h_j$ are holomorphic and the series converges normally locally. The coefficient $h_0$ is unrestricted and may vanish. Then $\cR_\la H$ is jointly holomorphic on the same domain, is odd near the axis, and satisfies
\[
(\cR_\la H)(0,z)=\partial_r(\cR_\la H)(0,z)=0,
\qquad \cL_\la\cR_\la H=H.
\]
The last identity at the axis is understood by removable continuation. The holomorphic-germ assumption applies to this continuation assertion; it is not an additional hypothesis for the preceding scalar estimates.
\end{lemma}
\begin{proof}
Direct differentiation gives $\cL_\la Z_\la=0$. Writing $\phi=Z_\la v$ transforms $\cL_\la\phi=H$ into
\[
(m_\la v_r)_r=m_\la H/Z_\la.
\]
The two integrations in \eqref{eq:Volterra} therefore give the differential identity away from the axis. The normalization and removability at the axis are verified below.

For real $\ep$ and $0<t\le s$,
\begin{equation}\label{eq:m-ratio}
\frac{m_\ep(t)}{m_\ep(s)}
=\frac{t^3(s^2+\ep^2)^2}{s^3(t^2+\ep^2)^2}\le\frac st.
\end{equation}
After exchanging the integrals in \eqref{eq:Volterra}, the scalar kernel is
\begin{align}
P_\ep(t;r)
&=m_\ep(t)\int_t^r\frac{ds}{m_\ep(s)}\notag\\
&=\frac{t^3}{(t^2+\ep^2)^2}
\left[\frac{r^2-t^2}{2}+2\ep^2\log\frac rt
+\frac{\ep^4}{2}(t^{-2}-r^{-2})\right].\label{eq:kernel}
\end{align}
In particular,
\begin{equation}\label{eq:P-upper}
0\le P_\ep(t;r)\le\frac{r^2-t^2}{2t}.
\end{equation}
Under \eqref{eq:lambda-comparison}, uniformly in $z$ and $0<t\le s$, we have
\[
|Z_\la(t,z)|\asymp Z_\ep(t),\qquad
\frac{|m_\la(t,z)|}{|m_\la(s,z)|}
\le C\frac{m_\ep(t)}{m_\ep(s)}.
\]
Consequently, for each fixed $z$,
\[
|\cR_\la H(r,z)|
\le CZ_\ep(r)\int_0^r
P_\ep(t;r)\frac{|H(t,z)|}{Z_\ep(t)}\dd t.
\]
This comparison includes both the exterior $Z_\la$ factor and the interior $1/Z_\la$ factor in \eqref{eq:Volterra}.

For the first inverse step it remains to bound
\[
J_\ep(r)=\int_0^r\frac{ds}{m_\ep(s)}\int_0^s m_\ep(t)\dd t.
\]
If $s\le\ep$, then $m_\ep(t)/m_\ep(s)\le C(t/s)^3$ for $0<t<s$, and hence
\[
\frac{1}{m_\ep(s)}\int_0^s m_\ep(t)\dd t\le Cs.
\]
If $s>\ep$, the estimates
\[
\int_0^\ep m_\ep(t)\dd t\le C\ep^2,
\qquad
m_\ep(t)\le C\ep^2/t\quad(t\ge\ep),
\qquad
m_\ep(s)^{-1}\le Cs/\ep^2
\]
give
\[
\frac{1}{m_\ep(s)}\int_0^s m_\ep(t)\dd t
\le Cs[1+\log(s/\ep)].
\]
Integrating these inequalities yields $J_\ep(r)\le Cr^2\ell_\ep(r)$, which proves \eqref{eq:first-resonance}. For $n\ge2$, use \eqref{eq:P-upper}, monotonicity of $\ell_\ep$, and the identity
\[
\int_0^r\frac{r^2-t^2}{2t}t^{2n-2}\dd t
=\frac{r^{2n}}{4n(n-1)}
\]
to prove \eqref{eq:radial-n}.

For a weight $w$ as in the statement, the same comparison at each fixed $z$ gives
\[
w(z)|\cR_\la H(r,z)|
\le CZ_\ep(r)\int_0^r
P_\ep(t;r)\frac{w(z)|H(t,z)|}{Z_\ep(t)}\dd t.
\]
Insert either weighted hypothesis into this pointwise inequality, use the corresponding positive-kernel bound just proved, and take the supremum in $z$. This proves the weighted assertions without differentiating or dividing by $w$, and in particular also at its zeros.

For complex $r\in\mathcal D_\ep$, uniformly under \eqref{eq:lambda-comparison},
\begin{equation}\label{eq:complex-denominator}
|r^2+\la^2|\asymp |r|^2+\ep^2.
\end{equation}
On $|\arg r|<\pi/32$, the terms $r^2$ and $\la^2$ lie in a fixed cone about the positive real axis with opening angle less than $\pi/4$. This proves \eqref{eq:complex-denominator} on the sector. On the full disk $|r|<\ep/4$, the lower bound follows instead from
\[
|\la|^2-|r|^2\ge\bigl((19/20)^2-1/16\bigr)\ep^2.
\]
Along a radial segment, write $r=se^{i\theta}$. The absolute values of $Z_\la$ are comparable to $Z_\ep(s)$, and the ratios of the absolute values of $m_\la$ are bounded by universal multiples of the corresponding real ratios. Parametrizing both integrals by their real lengths therefore proves every preceding estimate, including the weighted estimates, on $\mathcal D_\ep$.

It remains to verify the holomorphic assertion. Near each $(0,z_*)$, write
\[
Z_\la(r,z)=r\alpha(r^2,z),\qquad
\alpha(s,z)=\frac{2\la(z)}{s+\la(z)^2}.
\]
The function $\alpha$ is holomorphic and nonzero in a sufficiently small neighborhood. If $H(r,z)=r\widehat H(r^2,z)$, then
\[
m_\la=r^3\alpha(r^2,z)^2,
\qquad m_\la H/Z_\la=r^3\alpha(r^2,z)\widehat H(r^2,z).
\]
The inner integral in \eqref{eq:Volterra} consequently has the form $s^4\beta(s^2,z)$ with $\beta$ holomorphic. Dividing by $m_\la(s,z)$ gives
\[
s\frac{\beta(s^2,z)}{\alpha(s^2,z)^2},
\]
whose integral is $r^2\gamma(r^2,z)$ with $\gamma$ holomorphic. Thus
\[
\cR_\la H(r,z)=r^3\alpha(r^2,z)\gamma(r^2,z).
\]
This proves joint holomorphicity, oddness, and the asserted zero axis value and slope. Away from the axis, holomorphic parameter-dependent integration gives joint holomorphicity directly. The differential identity already proved for $r\ne0$ extends across the removable terms at the axis. The same calculation shows that an odd holomorphic source of order at least $r^{2j+1}$ is mapped to an odd holomorphic function of order at least $r^{2j+3}$, for every integer $j\ge0$.
\end{proof}

\subsection{Weighted Cauchy estimates and a nonlinear majorant}
For a holomorphic function on $D_h(z_0)$ set
\[
d(z)=h-|z-z_0|,\qquad
\norm{a}_p=\sup_{z\in D_h(z_0)}d(z)^p|a(z)|,
\qquad p=0,1,2,\ldots.
\]
These are the weights $w_p(z)=d(z)^p$ in Lemma~\ref{lem:radial}. The weighted holomorphic norms satisfy
\begin{equation}\label{eq:Nagumo}
\norm{a''}_{p+2}\le C_D(p+1)(p+2)\norm{a}_p,
\qquad
\norm{ab}_{p+q}\le\norm{a}_p\norm{b}_q,
\end{equation}
where one may take $C_D=4e^2$. Indeed, Cauchy's formula on the circle of radius $d(z)/(p+2)$ gives
\[
d(z)^{p+1}|a'(z)|
\le(p+2)\left(\frac{p+2}{p+1}\right)^p\norm{a}_p
\le2e(p+1)\norm{a}_p.
\]
Applying this inequality twice proves the derivative estimate in \eqref{eq:Nagumo}; its product estimate follows directly from the definition.

\begin{theorem}\label{thm:tube}
There are universal constants $\eta,C>0$ with the following property. Suppose \eqref{eq:lambda-comparison} holds and
\begin{equation}\label{eq:tube-smallness}
R^2\ell_\ep(R)/h^2<\eta.
\end{equation}
Then \eqref{eq:axial} has a real solution $f=F+\phi$ for
$0\le r<R/2$, $|z-z_0|<h/4$, with
\begin{equation}\label{eq:axis-data}
f(0,z)=0,\qquad f_r(0,z)=2/\la(z),
\end{equation}
and
\begin{equation}\label{eq:phi-bounds}
|\phi|+r|\phi_r|+h|\phi_z|
\le C Z_\ep(r)r^2\ell_\ep(r)/h^2.
\end{equation}
The associated Cartesian sphere map is smooth across the axis. For all $j,k\ge0$, on the same domain and for $\ep\le r<R/2$,
\begin{equation}\label{eq:tail-bounds}
|\partial_r^j\partial_z^k(\pi-f)|
\le C_{j,k}\ep r^{-j-1}h^{-k}.
\end{equation}
The solutions constructed here have canonical coefficients and agree on overlaps of their real domains whenever they use the same positive real-analytic function $\la$.
\end{theorem}
\begin{proof}
All coefficient estimates below hold for $r\in\mathcal D_{\ep,R}$, with the real majorants evaluated at $s=|r|$. Fix a universal constant $C_R\ge1$ for both radial estimates of Lemma~\ref{lem:radial}, including their weighted and complex versions. Integrating
$\partial_\la F=-2r/(r^2+\la^2)$ along the segment from $\ep$ to $\la(z)$ gives, by \eqref{eq:lambda-comparison} and \eqref{eq:complex-denominator},
\[
\norm{F(r,\cdot)-F_\ep(r)}_0\le C_FZ_\ep(s)
\]
for a universal $C_F\ge1$. The constant-scale function $F_\ep$ is independent of $z$, so \eqref{eq:Nagumo} implies
\begin{equation}\label{eq:source}
\norm{F_{zz}(r,\cdot)}_2\le2C_DC_FZ_\ep(s).
\end{equation}
Write
\[
\cN_F(v)=\frac{\sin(2F+2v)-\sin(2F)-2\cos(2F)v}{2r^2}.
\]
The equation for $\phi$ is $\cL_\la\phi=-F_{zz}-\phi_{zz}+\cN_F(\phi)$. Introduce an auxiliary complex variable $\tau$ only in the axial terms and seek
\begin{equation}\label{eq:tau-equation}
\cL_\la\phi=-\tau F_{zz}-\tau\phi_{zz}+\cN_F(\phi),
\qquad \phi=\sum_{n\ge1}\tau^n\phi_n.
\end{equation}
Equivalently, the equation for $F+\phi$ has coefficient $\tau$ in front of its second $z$ derivative. Define the coefficients recursively by radial integrals:
\begin{equation}\label{eq:coefficient-recursion}
\phi_1=-\cR_\la F_{zz},\qquad
\phi_n=\cR_\la[-\partial_{zz}\phi_{n-1}+N_n],\quad n\ge2.
\end{equation}
The nonlinear source in this formula is explicitly
\[
\begin{split}
b_k(r,z)&=\frac{2^{k-1}}{k!r^2}
\sin\left(2F(r,z)+\frac{k\pi}{2}\right),\qquad k\ge2,\\
N_n(r,z)&=\sum_{k=2}^n b_k(r,z)
\sum_{\substack{i_1+\cdots+i_k=n\\i_j\ge1}}
\prod_{j=1}^k\phi_{i_j}(r,z).
\end{split}
\]
Thus $\cN_F(v)=\sum_{k\ge2}b_kv^k$, and every index in $N_n$ is smaller than $n$.

We first verify inductively that these integrals define holomorphic functions, including at the axis. The source $F_{zz}$ has an odd holomorphic germ of order at least $r$, so Lemma~\ref{lem:radial} gives an odd holomorphic $\phi_1=O(r^3)$. Suppose, for $1\le i<n$, that near the axis
\[
\phi_i(r,z)=r^{2i+1}\widehat\phi_i(r^2,z),\qquad
\widehat\phi_i\text{ is holomorphic}.
\]
Then $\partial_{zz}\phi_{n-1}=r^{2n-1}\partial_{zz}\widehat\phi_{n-1}(r^2,z)$ is an odd holomorphic germ of order at least $r^{2n-1}$. A product with $i_1+\cdots+i_k=n$ in $N_n$ has order at least $r^{2n+k}$. If $k$ is even, $b_k$ is odd with at worst a simple pole, since its numerator is a multiple of $\sin(2F)=O(r)$. The product with $b_k$ is therefore odd and has order at least $r^{2n+k-1}$, which is at least $r^{2n+1}$. If $k\ge3$ is odd, $b_k$ is even with at worst a double pole, and the product is odd of order at least $r^{2n+k-2}$, again at least $r^{2n+1}$. Thus the full source in \eqref{eq:coefficient-recursion} has an odd holomorphic germ of order at least $r^{2n-1}$. Lemma~\ref{lem:radial} proves that $\phi_n$ is jointly holomorphic on $\mathcal D_{\ep,R}\times D_h(z_0)$ and has an odd germ of order at least $r^{2n+1}$. This establishes the analytic properties needed for each subsequent application of \eqref{eq:Nagumo}.

The rational identities
\[
\sin(2F)=\frac{4\la r(\la^2-r^2)}{(r^2+\la^2)^2},\qquad
\cos(2F)=\frac{\la^4-6\la^2r^2+r^4}{(r^2+\la^2)^2}
\]
and \eqref{eq:complex-denominator} give a universal $C_B\ge1$ such that
\[
\norm{b_k(r,\cdot)}_0\le\frac{C_B^k}{k!s^2},\qquad k\ge2.
\]
We claim that there is a sequence $(a_n)_{n\ge1}$ of at most exponential growth satisfying
\begin{equation}\label{eq:coefficient-bound}
\norm{\phi_n(r,\cdot)}_{2n}
\le a_n Z_\ep(s)s^{2n}\ell_\ep(s)^n.
\end{equation}
Set $P_n(s)=Z_\ep(s)s^{2n}\ell_\ep(s)^n$. Equations \eqref{eq:first-resonance} and \eqref{eq:source}, applied with the weight $d^2$, give
\[
\norm{\phi_1(r,\cdot)}_2\le2C_RC_DC_F P_1(s).
\]
Suppose \eqref{eq:coefficient-bound} holds for every index smaller than $n$. The derivative estimate in \eqref{eq:Nagumo} implies
\[
\begin{split}
\norm{\partial_{zz}\phi_{n-1}(r,\cdot)}_{2n}
&\le C_D(2n-1)(2n)a_{n-1}\\
&\qquad{}\times Z_\ep(s)s^{2n-2}\ell_\ep(s)^{n-1}.
\end{split}
\]
Applying \eqref{eq:radial-n} with weight $d^{2n}$, and using $\ell_\ep\ge1$, gives
\[
\norm{\cR_\la\partial_{zz}\phi_{n-1}(r,\cdot)}_{2n}
\le6C_RC_Da_{n-1}P_n(s),
\]
because
\[
\frac{(2n-1)(2n)}{n(n-1)}=4+\frac{2}{n-1}\le6.
\]
Here the Cauchy estimate is applied to the entire holomorphic function $\phi_{n-1}(r,z)$. It includes all axial dependence of its previously applied radial inverse operators.

For each composition $i_1+\cdots+i_k=n$, the product inequality in \eqref{eq:Nagumo} gives
\[
\left\|\prod_{j=1}^k\phi_{i_j}(r,\cdot)\right\|_{2n}
\le\left(\prod_{j=1}^k a_{i_j}\right)
Z_\ep(s)^k s^{2n}\ell_\ep(s)^n.
\]
Since $0\le Z_\ep(s)\le1$, the formula for $N_n$ yields
\[
\begin{split}
\norm{N_n(r,\cdot)}_{2n}
&\le Z_\ep(s)s^{2n-2}\ell_\ep(s)^n\\
&\qquad{}\times\sum_{k=2}^n\frac{C_B^k}{k!}
\sum_{\substack{i_1+\cdots+i_k=n\\i_j\ge1}}
\prod_{j=1}^k a_{i_j}.
\end{split}
\]
Another application of the weighted radial estimate gives
\[
\norm{\cR_\la N_n(r,\cdot)}_{2n}
\le\frac{C_RP_n(s)}{n(n-1)}
\sum_{k=2}^n\frac{C_B^k}{k!}
\sum_{\substack{i_1+\cdots+i_k=n\\i_j\ge1}}
\prod_{j=1}^k a_{i_j}.
\]
Choose one universal constant
\[
C\ge\max\{1,\,2C_RC_DC_F,\,6C_RC_D,\,C_R,\,C_B\},
\]
and define
\begin{align}
a_1&=C,\notag\\
a_n&=Ca_{n-1}
+C\sum_{k=2}^n\frac{C^k}{k!}
\sum_{\substack{i_1+\cdots+i_k=n\\ i_j\ge1}}
a_{i_1}\cdots a_{i_k}.\label{eq:majorant-recursion}
\end{align}
The preceding estimates and \eqref{eq:coefficient-recursion} prove \eqref{eq:coefficient-bound} by induction. In this upper bound the favorable factor $1/[n(n-1)]$ in the nonlinear term has been discarded using $1/[n(n-1)]\le1$.

To prove exponential growth, consider the scalar equation
\begin{equation}\label{eq:generating-function}
A=Ct+CtA+C\bigl(e^{CA}-1-CA\bigr).
\end{equation}
The holomorphic function
\[
\mathcal G(A,t)=A-Ct-CtA-C(e^{CA}-1-CA)
\]
satisfies $\mathcal G(0,0)=0$ and $\partial_A\mathcal G(0,0)=1$. The analytic implicit-function theorem supplies a holomorphic solution $A(t)$ near $0$ with $A(0)=0$. Comparing its Taylor coefficients with \eqref{eq:majorant-recursion} proves that
$A(t)=\sum_{n\ge1}a_nt^n$. On a closed disk of positive radius $t_*$ contained in its domain of holomorphicity, Cauchy's coefficient estimate gives $a_n\le Mt_*^{-n}$. Increasing a universal constant $C_0\ge1$ absorbs $M$ and gives $a_n\le C_0^n$ for every $n\ge1$.

For $z\in D_{h/2}(z_0)$ we have $d(z)\ge h/2$, so \eqref{eq:coefficient-bound} implies
\begin{equation}\label{eq:normal-convergence}
|\phi_n(r,z)|\le Z_\ep(s)
\left(\frac{4C_0s^2\ell_\ep(s)}{h^2}\right)^n,
\qquad s=|r|.
\end{equation}
Choose the universal constant in the theorem so that
\[
0<\eta\le\min\{1,(16C_0)^{-1}\}.
\]
By \eqref{eq:tube-smallness}, the quantity in parentheses in \eqref{eq:normal-convergence} is less than $1/4$ whenever $s<R$. Therefore
\[
\phi(r,z,\tau)=\sum_{n\ge1}\tau^n\phi_n(r,z)
\]
converges normally on
\[
\mathcal D_{\ep,R}\times D_{h/2}(z_0)\times\{\tau\in\C:|\tau|<2\}.
\]
Its sum is jointly holomorphic there, and every derivative series converges normally on compact subsets. At $\tau=1$, summation of the geometric majorant gives
\[
|\phi(r,z,1)|
\le CZ_\ep(s)\frac{s^2\ell_\ep(s)}{h^2}.
\]

On a compact subset with $r\ne0$, normal convergence permits termwise differentiation and substitution into the entire sine function. Both sides of \eqref{eq:tau-equation} are holomorphic in $\tau$ for $|\tau|<2$, and their Taylor coefficients agree by \eqref{eq:coefficient-recursion} and the finite formula for $N_n$. The identity theorem therefore proves \eqref{eq:tau-equation} on that domain. At $\tau=1$, it is exactly \eqref{eq:axial} for $f=F+\phi$. Each coefficient is real at real $(r,z)$, because the scale, sources, and radial integrations are real there. Hence $f$ is a real solution. In the remaining estimates we suppress $\tau=1$ in the notation.

Fix real $0<r<R/2$ and $|z-z_0|<h/4$, and use the complex polydisk
\[
\mathcal P_{r,z}=D_{r/128}(r)\times D_{h/8}(z).
\]
Its radial disk lies in the sector $|\arg\zeta|<\pi/32$ and in $D_R(0)$; its axial disk lies in $D_{h/2}(z_0)$. On its radial disk,
\[
127r/128<|\zeta|<129r/128.
\]
The explicit formulas for $Z_\ep$ and $\ell_\ep$ therefore bound their values at $|\zeta|$ by universal multiples of their values at $r$. The preceding complex bound on $\phi$ gives
\[
\sup_{\mathcal P_{r,z}}|\phi|
\le CZ_\ep(r)\frac{r^2\ell_\ep(r)}{h^2}.
\]
Cauchy's formula in each variable proves \eqref{eq:phi-bounds} for $r>0$.

For the axis, put $a=\min\{R/2,\ep/8\}>0$. The full disk $D_a(0)$ is contained in $\mathcal D_{\ep,R}$. The normally convergent series is therefore holomorphic on
$D_a(0)\times D_{h/2}(z_0)$, including at $r=0$. The inductive germ calculation showed that every $\phi_n$ is odd and has zero linear coefficient. Normal convergence of derivatives on smaller polydisks implies
\[
\phi(0,z)=\partial_r\phi(0,z)=0.
\]
Thus $f$ is odd and holomorphic near the axis and satisfies \eqref{eq:axis-data}. It also follows that $\phi_z(0,z)=0$, so \eqref{eq:phi-bounds} holds at $r=0$ by continuity. Near every point of this axis there are jointly holomorphic functions $b(s,z)$ and $c(s,z)$ such that
\[
\frac{\sin f(r,z)}{r}=b(r^2,z),\qquad
\cos f(r,z)=c(r^2,z).
\]
Indeed, the first numerator is odd and holomorphic, and the second expression is even and holomorphic. In Cartesian coordinates the associated sphere map is
\[
q_f(x_1,x_2,z)=
\bigl(x_1b(x_1^2+x_2^2,z),\,
x_2b(x_1^2+x_2^2,z),\,
c(x_1^2+x_2^2,z)\bigr).
\]
It is real analytic near the regular axis. Its classical sphere tension vanishes for $r>0$ by the equation already proved and hence vanishes at the axis by continuity. This proves the asserted Cartesian smoothness and the extension of the equation there.

For the tail estimate, fix a real point with $\ep\le r<R/2$ and $|z-z_0|<h/4$, and use the same polydisk $\mathcal P_{r,z}$. Analytic continuation of the positive-real identity for the selected branch gives
\[
\pi-F(\zeta,\omega)
=2\int_0^1
\frac{\zeta\la(\omega)}{\zeta^2+t^2\la(\omega)^2}\dd t,
\qquad (\zeta,\omega)\in\mathcal P_{r,z}.
\]
For $0\le t\le1$, the two denominator terms lie in a common fixed cone about the positive real axis, so
\[
|\zeta^2+t^2\la(\omega)^2|
\ge c\bigl(|\zeta|^2+t^2\ep^2\bigr).
\]
Together with $|\la(\omega)|\le21\ep/20$ and $|\zeta|\asymp r$, this bounds the integral by $C\ep/r$, uniformly on the polydisk. Also $Z_\ep(|\zeta|)\le C\ep/r$, and the complex bound on $\phi$, together with \eqref{eq:tube-smallness}, gives the same estimate for $\phi$. Consequently
\[
\sup_{\mathcal P_{r,z}}|\pi-f|\le C\ep/r.
\]
The two-variable Cauchy estimate yields, for every $j,k\ge0$,
\[
|\partial_r^j\partial_z^k(\pi-f)(r,z)|
\le Cj!k!\,128^j8^k\,
\ep r^{-j-1}h^{-k}.
\]
This proves \eqref{eq:tail-bounds} with constants depending only on $j,k$ and the universal constants fixed above.

Finally, consider two applications of the theorem that use holomorphic extensions of the same positive real-analytic scale on overlapping real axial intervals. By the identity theorem, these extensions agree in a complex neighborhood of every point of the real overlap. Their functions $F$, normalized by $F(0,z)=0$, and their functions $Z_\la,m_\la$ then agree along every common radial segment. Formula \eqref{eq:coefficient-recursion} gives the same $\phi_1$ in both constructions. If the coefficients agree through index $n-1$ in the common local complex domain, their axial derivatives agree there, as does the finite expression for $N_n$. The identical radial integrals from zero therefore give the same $\phi_n$. Induction proves equality of all coefficients. Since both series converge at each point of the intersection of the asserted real domains, their sums agree there. The parameters $z_0,h,\ep$ enter only the estimates and domain choices, not this recursion. This proves the overlap assertion for the constructed solutions without requiring uniqueness among all solutions with the same singular Cauchy data.
\end{proof}

\subsection{Flat scales on the complementary intervals}\label{sec:cutoff}
We choose the scale so that Theorem~\ref{thm:tube} applies uniformly on every gap. Exponential flatness controls the error after cutoff, while the polynomial radius leaves room for the exact harmonic core.

Fix the compact null set $E\subset\T$ in Theorem~\ref{thm:main}. Each component of $\T\setminus E$ is an open arc that can be lifted to an interval $(a,b)$ of length at most $L$. On that interval define
\begin{equation}\label{eq:flat-scale}
A(z)=\frac1{(z-a)(b-z)},\qquad
\la_\delta(z)=\delta e^{-A(z)},
\end{equation}
and set $\la_\delta=0$ on $E$. Fix $0<\delta<e^{-1}$ and put
\begin{equation}\label{eq:radii}
B=1+|\log\delta|,\qquad
h(z)=c_h(1+A(z))^{-2},\qquad
\rho(z)=\frac\kappa B(1+A(z))^{-4}.
\end{equation}
The positive constants $c_h,\kappa$ will be fixed independently of $\delta$ and of the individual gap. Here and below $a\ll b$ means $a\le c b$, where the fixed positive constant $c$ can be made sufficiently small by decreasing $\kappa$. For $\la/\rho$, smallness is obtained instead by decreasing $\delta$ after these constants have been fixed. The radii $h$ and $\rho$ are used on $\T\setminus E$; no smooth extension of these radii across $E$ is assumed.

\begin{lemma}\label{lem:flat-scale}
The scale satisfies
\[
\la_\delta\in C^\infty(\T),\qquad
\partial_z^j\la_\delta\big|_E=0\quad(j\ge0).
\]
Uniformly over the complementary intervals,
\begin{equation}\label{eq:gap-derivatives}
|A^{(j)}|\le C_j(1+A)^{j+1},\qquad
|\la_\delta^{(j)}|\le C_j\la_\delta(1+A)^{2j}.
\end{equation}
The constants in \eqref{eq:radii} can be chosen so that each disk $D_{h(z_0)}(z_0)$ satisfies \eqref{eq:lambda-comparison} with $\ep=\la_\delta(z_0)$, and Theorem~\ref{thm:tube} yields one exact solution on
\begin{equation}\label{eq:cusp-domain}
\mathcal C_\delta=\{(r,z):z\notin E,\ 0\le r<2\rho(z)\}.
\end{equation}
Moreover, $\rho\ll h$, $\rho\ll\dist(z,E)$, and $\sup_{z\notin E}\la_\delta(z)/\rho(z)\to0$ as $\delta\downarrow0$.
\end{lemma}
\begin{proof}
Write $u=z-a$, $v=b-z$, and $s=b-a=u+v\le L$. The identity
\[
A(z)=\frac1s\left(\frac1u+\frac1v\right)
\]
gives, for every integer $j\ge0$,
\[
|A^{(j)}(z)|
\le\frac{j!}{s}\bigl(u^{-j-1}+v^{-j-1}\bigr)
\le j!s^j A(z)^{j+1}
\le j!L^j(1+A(z))^{j+1}.
\]
The middle inequality follows from $u^{j+1}+v^{j+1}\le(u+v)^{j+1}$. For $j\ge1$, every term in the derivative of $e^{-A}$ is a constant times
\[
e^{-A}\prod_{k=1}^j\bigl(A^{(k)}\bigr)^{m_k},
\qquad \sum_{k=1}^j k m_k=j.
\]
The corresponding exponent of $1+A$ is
$\sum(k+1)m_k=j+\sum m_k\le2j$.
This proves \eqref{eq:gap-derivatives}, including $j=0$.

If $d=\dist(z,E)=\min(u,v)$, then $A\ge1/(Ld)$. For every $j,M\ge0$, the estimates just proved imply
\[
d^{-M}|\la_\delta^{(j)}(z)|
\le C_{j,M}\delta A^M(1+A)^{2j}e^{-A}
\longrightarrow0\qquad(d\downarrow0),
\]
uniformly over the gaps. Extending each derivative by zero on $E$ therefore gives a continuous function. At any $z_*\in E$, for $z\notin E$ tending to $z_*$ in a local lift, $d(z,E)\le|z-z_*|$ and hence
\[
\frac{|\la_\delta^{(j)}(z)|}{|z-z_*|}
\le\frac{|\la_\delta^{(j)}(z)|}{d(z,E)}\longrightarrow0.
\]
For $z\in E$ the numerator of the corresponding difference quotient is zero. Thus each proposed derivative has derivative zero at every point of $E$, while on the gaps its derivative is the next proposed derivative. Induction proves that the extension is smooth and flat on $E$.

Fix a gap and a real center $z_0$ in it. Write
\[
u_0=z_0-a,\qquad v_0=b-z_0,\qquad
A_0=(u_0v_0)^{-1},\qquad d_0=\min(u_0,v_0).
\]
Then $d_0^{-1}\le LA_0$ and $A_0/(1+A_0)^2\le1/4$. Choose, once and for all,
\[
c_h=\frac{\log(21/20)}{6L},\qquad h_0=c_h(1+A_0)^{-2}.
\]
In particular,
\[
\frac{h_0}{d_0}
\le Lc_h\frac{A_0}{(1+A_0)^2}
\le\frac{Lc_h}{4}<\frac14.
\]
For $\zeta=z_0+w\in D_{h_0}(z_0)$ we have
$|u_0+w|\ge3u_0/4$ and $|v_0-w|\ge3v_0/4$.
The rational function $A(\zeta)=((\zeta-a)(b-\zeta))^{-1}$ is consequently holomorphic on this disk, and
\[
\begin{aligned}
|A(\zeta)|&\le\frac{16}{9}A_0,\\
|A'(\zeta)|
&\le\frac{64}{27}\left(\frac1{u_0^2v_0}+\frac1{u_0v_0^2}\right)
=\frac{64}{27}(u_0+v_0)A_0^2<3LA_0^2.
\end{aligned}
\]
Integration along the segment from $z_0$ to $\zeta$ gives
\[
|A(\zeta)-A_0|\le3LA_0^2h_0
\le3Lc_h=\tfrac12\log(21/20).
\]
For the holomorphic continuation $\la_\delta(\zeta)=\delta e^{-A(\zeta)}$ and $\ep=\delta e^{-A_0}$, it follows that
\[
\left|\frac{\la_\delta(\zeta)}\ep-1\right|
=|e^{-(A(\zeta)-A_0)}-1|
\le e^{|A(\zeta)-A_0|}-1
\le\sqrt{21/20}-1<\frac1{20}.
\]
This is \eqref{eq:lambda-comparison}. The factor $\delta$ cancels from the estimate, and the choice of $c_h$ depends only on the fixed circle length $L$.

Let $\eta$ be the universal constant in Theorem~\ref{thm:tube}. After fixing $c_h$, choose
\[
0<\kappa\le\min\left\{
\frac14,\frac{c_h}{64},\frac1{64L},\frac{c_h\sqrt\eta}{16}
\right\}.
\]
This choice depends only on $L$ and $\eta$. Since $B>2$, it gives
\[
\begin{aligned}
\frac\rho h&=\frac{\kappa}{c_hB(1+A)^2}
\le\frac{\kappa}{2c_h}\le\frac1{128},\\
\frac\rho d&\le\frac{\kappa LA}{B(1+A)^4}
\le\frac{\kappa L}{4B}\le\frac1{512}.
\end{aligned}
\]
Here $A/(1+A)^4\le A/(1+A)^2\le1/4$. These bounds quantify both geometric comparisons in the statement.

To estimate the logarithmic factor, observe that
\[
4\rho=\frac{4\kappa}{B}(1+A)^{-4}<1,
\qquad 0<\la=\delta e^{-A}<1.
\]
Consequently $1+4\rho/\la\le1+1/\la\le2/\la$, and hence
\[
\ell_\la(4\rho)
\le1+\log2-\log\la
=A+B+\log2\le2(A+B).
\]
The last inequality uses $A+B>2>\log2$. In particular, the first estimate in the following display holds with the numerical constant $2$:
\begin{equation}\label{eq:log-radius}
\ell_{\la(z)}(4\rho(z))\le C(B+A(z)),\qquad
\frac{(4\rho)^2\ell_{\la}(4\rho)}{h^2}
\le C\frac{\kappa^2(B+A)}{B^2(1+A)^4}\le C\kappa^2.
\end{equation}
More explicitly, the second estimate satisfies
\[
\begin{aligned}
\frac{(4\rho)^2\ell_\la(4\rho)}{h^2}
&\le\frac{32\kappa^2}{c_h^2}
\frac{B+A}{B^2(1+A)^4}\\
&\le\frac{32\kappa^2}{c_h^2B(1+A)^3}
\le\frac{16\kappa^2}{c_h^2}
\le\frac\eta{16}<\eta.
\end{aligned}
\]
We used $B+A\le B(1+A)$ in the second line. Thus \eqref{eq:tube-smallness} holds with
$R=4\rho(z_0)$, $h=h(z_0)$, and $\ep=\la_\delta(z_0)$,
uniformly in $\delta$ and the gap.

Theorem~\ref{thm:tube} produces a solution on each rectangle
\[
0\le r<2\rho(z_0),\qquad |z-z_0|<h(z_0)/4.
\]
Every point $(r,z)$ of \eqref{eq:cusp-domain} belongs to such a rectangle by taking $z_0=z$. The canonical-coefficient assertion of that theorem identifies the solutions on every overlap. The real axial intervals stay within their original gaps because $h(z_0)<d_0/4$. The solutions therefore define a single exact solution on $\mathcal C_\delta$, without enlarging the output domain of Theorem~\ref{thm:tube}.

Finally, elementary differentiation shows that
\[
\sup_{A\ge0}e^{-A}(1+A)^4=256e^{-3},
\]
with the maximum attained at $A=3$. Therefore
\[
\sup_{z\notin E}\frac{\la_\delta(z)}{\rho(z)}
\le\frac{256e^{-3}}\kappa\delta B\longrightarrow0.
\]
After the fixed choices of $c_h,\kappa$, choose $\delta_0>0$ so that
$\la_\delta(z)\le\rho(z)/100$ for all $z\notin E$ and $0<\delta<\delta_0$.
On the cutoff annuli used below this ensures $r\ge\ep$, as required by \eqref{eq:tail-bounds} with center $z_0=z$. Those annuli also lie strictly within $r<2\rho(z_0)$, so no further restriction depending on the gap is needed.
\end{proof}

\subsection{Proof of Proposition~\ref{prop:construction}:  Cutoff and global assembly}\label{sec:assembly}
\begin{proof}[Proof of Proposition~\ref{prop:construction}]
Lemma~\ref{lem:flat-scale} gives the canonical solution $f$ on \eqref{eq:cusp-domain}, with the axis data and estimates from Theorem~\ref{thm:tube}. We now perform the cutoff and verify that it defines a global map on the torus.

Use the cutoff $\chi$ fixed in Proposition~\ref{prop:construction}. For $(p(\xi),z)\in\Omega_\delta(2)$ with $\xi=r(\cos\theta,\sin\theta)\ne0$, define
\begin{equation}\label{eq:angle-cutoff}
\widetilde f=\pi+\chi(r/\rho(z))(f-\pi),\qquad
q_\delta(p(\xi),z)=(\sin\widetilde f\cos\theta,
\sin\widetilde f\sin\theta,\cos\widetilde f).
\end{equation}
Set $q_\delta=\Sp$ on $X\setminus\bigl(\Omega_\delta(2)\cup S_E\bigr)$ and $q_\delta|_{S_E}=v_E$. At $(0,0,z)$ with $z\in\T\setminus E$, the analytic extension gives
\[
q_\delta(0,0,z)=\Np.
\]

We verify that these definitions give a map on $X$. For $p:\R^2\to\T^2=\R^2/(L\mathbb Z)^2$ and $D=B_1^{\R^2}(0)$ fixed above, $L=10$ implies that $p|_D$ is an isometric embedding. Formula \eqref{eq:angle-cutoff} defines $q_\delta(p(\xi),z)$ on $\Omega_\delta(2)\setminus\{(0,0,z):z\in\T\setminus E\}$; the preceding extensions cover all remaining points of $p(D)\times\T$. On $(\T^2\setminus p(D))\times\T$ we have $q_\delta=\Sp$.
These definitions agree on a fixed neighborhood of the chart boundary. Indeed,
$\rho(z)\le\kappa/B\le\kappa/2$, and $\chi(t)=0$ for $t\ge3/4$, so
\[
q_\delta(p(\xi),z)=\Sp
\quad\hbox{whenever}\quad
|\xi|\ge3\kappa/8.
\]
For $z\in E$ the same assertion follows from the exterior convention, since $|\xi|>0$. Our fixed choice gives $3\kappa/8\le3/32<1$. The map is therefore constant on an open collar of $\partial p(D)$, and the two definitions agree there with all derivatives. More precisely,
\[
\overline{\{x\in X:q_\delta(x)\ne\Sp\}}
\subset p\bigl(\overline B_{3\kappa/8}^{\R^2}(0)\bigr)\times\T,
\]
whose transverse factor is compactly contained in $p(D)$.

The resulting map is defined on $X=\T^2\times\T$. Its pullback to $\R^3$ is $L$-periodic in the two transverse variables because these variables were defined on the quotient. For every $k\in\mathbb Z$, translating a lifted gap gives
\[
(A,\la_\delta,\rho)_{(a+kL,b+kL)}(z+kL)
=(A,\la_\delta,\rho)_{(a,b)}(z).
\]
The canonical coefficient recursion is unchanged, so $f_{(a+kL,b+kL)}(r,z+kL)=f_{(a,b)}(r,z)$. Hence the definition is $L$-periodic in $z$ and introduces no singularities at coordinate boundaries.

In what follows we omit the subscript $\delta$ and take $0<\delta<\delta_0$. Write
\[
\mathcal A_\delta=\{(r,\theta,z):z\notin E,\ \rho(z)/2\le r\le3\rho(z)/4\}
\]
for the cutoff annuli. The map is smooth on $X\setminus S_E$. At a point $r>0,z\in E$, it is locally constant because $\rho(z')\to0$ as $z'\to E$ through $\T\setminus E$; this fact uses only the radius formula, not smoothness of $\rho$ across $E$.

On $\Omega_\delta(1/2)$ the cutoff is one, and $q_\delta$ is the exact harmonic map. On $X\setminus\bigl(\Omega_\delta(3/4)\cup S_E\bigr)$ it equals $\Sp$, while $q_\delta(0,0,z)=\Np$ for $z\in\T\setminus E$. These observations and the scale separation in Lemma~\ref{lem:flat-scale} prove the proposition.
\end{proof}

\section{Verification of the construction}\label{sec:verification}
In this section, we verify that the map of Proposition~\ref{prop:construction} satisfies every hypothesis of Theorem~\ref{prop:repair}. The cutoff estimates give smooth residuals and finite energy. A capacity argument gives the weak equation. A separate tube-flux calculation gives the full stress identity and excludes a distribution supported on $S_E$. Indeed, we prove

\begin{proposition}\label{prop:verification}
Let $q_\delta$ be as in Proposition~\ref{prop:construction}, and put $B=1+|\log\delta|$. Then
\[
q_\delta\in W^{1,2}(X,\sph),\qquad
\sup_{0<\delta<\delta_0}E(q_\delta;X)<\infty,\qquad
\sing q_\delta=S_E,\qquad \capacity_2(S_E)=0,
\]
and every point of $S_E$ is an essential discontinuity. The classical fields
\[
e_\delta=\Delta q_\delta+|Dq_\delta|^2q_\delta,\qquad
(\cF_\delta)_j=-2e_\delta\cdot\partial_jq_\delta
\quad\text{on }X\setminus S_E
\]
extend smoothly and flatly across $S_E$. On all of $X$,
\[
\Delta q_\delta+|Dq_\delta|^2q_\delta=e_\delta,\quad
\Div T_{q_\delta}=\cF_\delta\quad\text{in }\mathcal D'(X),\qquad
\int_X\cF_\delta=0.
\]
For every integer $m\ge0$,
\begin{equation}\label{eq:verification-bounds}
\|e_\delta\|_{C^m}\le C_m\delta B^{m+3},\qquad
\|\cF_\delta\|_{C^m}\le C_m\delta^2B^{m+5}.
\end{equation}
Consequently the smallness condition of Theorem~\ref{prop:repair} holds for sufficiently small $\delta$.
\end{proposition}

We omit the subscript $\delta$ where there is no ambiguity and use the annuli $\mathcal A_\delta$ defined in Subsection~\ref{sec:assembly}.

\subsection{Smoothness and smallness of the cutoff residuals}
The exact core has zero tension. All residuals vanish outside the cutoff annuli, where the map is close to the south pole. We estimate them in Cartesian coordinates.

\begin{lemma}\label{lem:annulus}
For every integer $m\ge0$ there is $C_m$, independent of sufficiently small $\delta$ and of the gap, such that on $\mathcal A_\delta$
\begin{align}
|D^m(q-\Sp)|&\le C_m\la\rho^{-m-1},\label{eq:q-annulus}\\
|D^me|&\le C_m\la\rho^{-m-3},\label{eq:e-annulus}\\
|D^m(e\cdot Dq)|&\le C_m\la^2\rho^{-m-5},\label{eq:force-annulus}
\end{align}
where $e=\Delta q+|Dq|^2q$ is initially computed on $X\setminus S_E$ and $D^m$ denotes physical Cartesian derivatives. Both $e$ and $e\cdot Dq$ extend smoothly and flatly across $S_E$, with
\begin{equation}\label{eq:residual-Cm}
\norm e_{C^m}\le C_m\delta B^{m+3},\qquad
\norm{e\cdot Dq}_{C^m}\le C_m\delta^2B^{m+5}.
\end{equation}
\end{lemma}
\begin{proof}
On $r\asymp\rho$, Theorem~\ref{thm:tube} gives
\[
|\partial_r^j\partial_z^k(\pi-f)|
\le C_{j,k}\la\rho^{-j-1}h^{-k}
\le C_{j,k}\la\rho^{-j-k-1}.
\]
Here $\la\ll\rho\ll h$. Differentiating \eqref{eq:radii} also gives
$|\partial_z^k\rho|\le C_k\rho h^{-k}$.
Thus, for $j+k=m$,
\[
|\partial_r^j\partial_z^k\chi(r/\rho)|\le C_m\rho^{-m}.
\]
Since $r\asymp\rho$, passage to Cartesian derivatives gives the same estimate $|D^m\chi(r/\rho)|\le C_m\rho^{-m}$. The chain and product rules in \eqref{eq:angle-cutoff}, using $\la/\rho\le1$, prove \eqref{eq:q-annulus}.

The Laplacian term in $D^me$ is bounded by $C_m\la\rho^{-m-3}$. The differentiated quadratic term is at most a constant times $\la^2\rho^{-m-4}$, with possible additional factors $\la/\rho\le1$. This proves \eqref{eq:e-annulus}. Multiplication by the differentiated gradient of $q$ gives \eqref{eq:force-annulus}.

The exact equation in the core and constancy in the exterior give
\[
e=0,\qquad e\cdot Dq=0
\quad\text{on }(X\setminus S_E)\setminus\mathcal A_\delta.
\] On the annuli the distance to $S_E$ is comparable to $\dist(z,E)$ because $\rho\ll\dist(z,E)$. Substitution of \eqref{eq:flat-scale}--\eqref{eq:radii} into \eqref{eq:e-annulus} and \eqref{eq:force-annulus} gives, for all integers $m,M\ge0$,
\[
\lim_{\substack{\dist(x,S_E)\to0\\x\notin S_E}}
\frac{|D^me(x)|+|D^m(e\cdot Dq)(x)|}{\dist(x,S_E)^M}=0.
\]
The zero extensions are therefore smooth and flat. Finally,
\[
\la\rho^{-m-3}
=\delta e^{-A}(B/\kappa)^{m+3}(1+A)^{4m+12},
\]
and the analogous force bound contains $\delta^2e^{-2A}(B/\kappa)^{m+5}(1+A)^{4m+20}$. Taking suprema in $A$ proves \eqref{eq:residual-Cm}.
\end{proof}

\subsection{Energy, weak extension, and essential discontinuities}\label{sec:energy}
The slice-energy bounds give finite energy, and zero capacity permits extension of the tension equation. The two pole values identify the singular set independently of the values assigned on a null set.

\begin{proposition}\label{prop:q-properties}
For sufficiently small $\delta_0>0$,
\[
q_\delta\in W^{1,2}(X,\sph),\qquad
\sup_{0<\delta<\delta_0}\int_X|Dq_\delta|^2\dd x<\infty,
\]
and
\begin{equation}\label{eq:weak-tension}
\Delta q_\delta+|Dq_\delta|^2q_\delta=e_\delta
\end{equation}
on all of $X$ in distributions. Its singular set, in the sense of essential discontinuity, is exactly $S_E$.
\end{proposition}
\begin{proof}
Fix $0<\delta<\delta_0$, with $\delta_0<e^{-1}$ sufficiently small for the construction in Proposition~\ref{prop:construction}, and suppress the subscript $\delta$ on $q$, $e$, and $\la$. All constants below may depend on the fixed construction data, including $L$, $c_h$, $\kappa$, and the cutoff function, but are independent of $\delta$ and of the individual complementary interval of $E$. Until the Sobolev extension is established, derivatives of $q$ are understood classically on $X\setminus S_E$.

\emph{Core estimates.}
Write
\[
\mathcal K_\delta
=\{(r,\theta,z):z\notin E,\ 0<r<\rho(z)/2\}.
\]
On this set the cutoff is inactive, so $q=q_f$ with $f=F+\phi$. At a fixed evaluation point $z\notin E$, apply Theorem~\ref{thm:tube} with center $z_0=z$, $\ep=\la(z)$, $h=h(z)$, and $R=4\rho(z)$. The agreement of the local solutions on overlaps identifies their derivatives with those of the assembled solution. In particular, all derivatives in the estimates below are taken within a fixed local solution; no differentiation of a moving choice of center is involved.

Set, as functions of $(r,z)$,
\[
Z(r,z)=\frac{2\la(z)r}{r^2+\la(z)^2},\qquad
\eta_r(z)=\frac{r^2\ell_{\la(z)}(r)}{h(z)^2}.
\]
The explicit choices in Lemma~\ref{lem:flat-scale}, together with the monotonicity of $\ell_\la$, imply on $\mathcal K_\delta$ that
\[
\frac rh\le\frac1{256},\qquad
\eta_r
\le\left(\frac r{4\rho}\right)^2
\frac{(4\rho)^2\ell_\la(4\rho)}{h^2}
\le\frac\eta{1024}\le\frac1{1024}.
\]
Here $0<\eta\le1$ is the constant in Theorem~\ref{thm:tube}. If $C_\phi$ denotes the uniform constant in \eqref{eq:phi-bounds}, that estimate gives separately
\[
|\phi|\le C_\phi Z\eta_r,\qquad
|\phi_r|\le C_\phi\frac Zr\eta_r,\qquad
|\phi_z|\le C_\phi\frac Zh\eta_r.
\]

For completeness, the logarithmic derivative of the bubble scale has a bound uniform in the gap. On a lifted gap $(a,b)$, put $u=z-a$ and $v=b-z$. Since $A=(uv)^{-1}$ and $u+v=b-a\le L$,
\[
A'=(u-v)A^2,\qquad
\frac{|\la'|}{\la}=|A'|
\le LA^2\le L(1+A)^2=\frac{Lc_h}{h}.
\]
Differentiating $F=2\arctan(r/\la(z))$ therefore gives
\begin{equation}\label{eq:F-derivatives}
F_r=Z/r,\qquad F_z=-Z\la'/\la,
\qquad |\la'|/\la\le Lc_h/h.
\end{equation}
Also $\sin F=Z$. Since sine is $1$-Lipschitz on the real line, the preceding estimates yield
\[
\begin{aligned}
|f_r|&\le(1+C_\phi\eta_r)\frac Zr,\\
|f_z|&\le(Lc_h+C_\phi\eta_r)\frac Zh,\\
\frac{|\sin f|}{r}
&\le\frac{|\sin F|+|\phi|}{r}
\le(1+C_\phi\eta_r)\frac Zr.
\end{aligned}
\]
To translate these into the full Cartesian gradient, define
\[
e_f=(\cos f\cos\theta,\cos f\sin\theta,-\sin f),\qquad
e_\theta=(-\sin\theta,\cos\theta,0).
\]
These vectors are orthonormal, and
$q_r=f_re_f$, $q_z=f_ze_f$, $q_\theta=\sin f\,e_\theta$. Thus the Euclidean cylindrical-coordinate formula gives
\[
\begin{aligned}
|Dq|^2
&=|q_r|^2+|q_z|^2+r^{-2}|q_\theta|^2
=f_r^2+f_z^2+\frac{\sin^2 f}{r^2}\\
&\le\left[2(1+C_\phi\eta_r)^2
+(Lc_h+C_\phi\eta_r)^2\left(\frac rh\right)^2\right]
\frac{Z^2}{r^2}
\le C_*\frac{Z^2}{r^2},
\end{aligned}
\]
where one may take $C_*=2(1+C_\phi)^2+(Lc_h+C_\phi)^2$. In particular,
\begin{equation}\label{eq:core-gradient}
|Dq|\le CZ/r,
\end{equation}
with $C=\sqrt{C_*}$ independent of $\delta$ and the gap. This explicitly controls the radial, axial, and angular contributions to the energy.

\emph{Integration of the energy.}
At fixed $z\notin E$, the primitive of $4\la^2r/(r^2+\la^2)^2$ is $-2\la^2/(r^2+\la^2)$. Consequently
\begin{equation}\label{eq:slice-energy}
\int_0^\rho \frac{Z^2}{r^2}r\dd r
=\int_0^\rho\frac{4\la^2r}{(r^2+\la^2)^2}\dd r
=\frac{2\rho^2}{\rho^2+\la^2}\le2.
\end{equation}
The core ends at $\rho/2$; its nonnegative radial energy bound is at most the scalar integral in \eqref{eq:slice-energy}. Using the volume element $\dd x=r\dd r\dd\theta\dd z$ and Tonelli's theorem, we obtain
\[
\begin{aligned}
\int_{\mathcal K_\delta}|Dq|^2\dd x
&\le2\pi C_*\int_{\T\setminus E}
\int_0^{\rho(z)/2}\frac{Z(r,z)^2}{r^2}r\dd r\dd z\\
&\le4\pi C_*\cL^1(\T\setminus E)
\le4\pi C_*L.
\end{aligned}
\]

On the cutoff annuli $\mathcal A_\delta$, the case $m=1$ of \eqref{eq:q-annulus} gives $|Dq|\le C_1\la\rho^{-2}$. Therefore, for each $z\notin E$,
\[
\begin{aligned}
\int_0^{2\pi}\int_{\rho/2}^{3\rho/4}|Dq|^2r\dd r\dd\theta
&\le2\pi C_1^2\la^2\rho^{-4}
\left[\frac{r^2}{2}\right]_{\rho/2}^{3\rho/4}\\
&=\frac{5\pi C_1^2}{16}\frac{\la^2}{\rho^2}.
\end{aligned}
\]
The scale formulas \eqref{eq:flat-scale}--\eqref{eq:radii} imply
\[
\frac{\la^2}{\rho^2}
=\frac{\delta^2B^2}{\kappa^2}e^{-2A}(1+A)^8
\le\frac{4^8e^{-6}}{\kappa^2}\delta^2B^2.
\]
Indeed, the logarithmic derivative of $e^{-2A}(1+A)^8$ is $-2+8/(1+A)$, so its maximum for $A\ge0$ is attained at $A=3$. Moreover, if $t=-\log\delta>1$, then
\[
\delta B=(1+t)e^{-t}\le2e^{-1},
\]
because the derivative of $(1+t)e^{-t}$ is $-te^{-t}<0$. Outside the core and annuli, $q$ is constant at every off-axis point; the axis has three-dimensional measure zero. Thus
\[
\begin{aligned}
\int_{X\setminus S_E}|Dq|^2\dd x
&\le4\pi C_*L
+\frac{5\pi C_1^2L}{16}
\frac{4^8e^{-6}}{\kappa^2}\delta^2B^2\\
&\le4\pi C_*L
+\frac{5\pi C_1^2L}{16}
\frac{4^9e^{-8}}{\kappa^2}<\infty.
\end{aligned}
\]
Both $\kappa$ and $C_1$ were fixed independently of $\delta$ and the gap. The integration over all gaps uses only their total length, at most $L$, so the displayed bound is uniform even when there are infinitely many gaps.

\emph{Sobolev extension and the weak equation.}
The Lipschitz inclusion $z\mapsto(0,0,z)$ and the capacity criterion in Subsection~\ref{sec:capacity} give
\[
\cL^1(E)=0\quad\Longrightarrow\quad
\cH^1(S_E)=0\quad\Longrightarrow\quad\capacity_2(S_E)=0.
\] The map $q$ is bounded and smooth off $S_E$, and the preceding calculation proves that its classical gradient is square integrable there. The first assertion of Lemma~\ref{lem:capacity} therefore gives a global Sobolev extension whose weak gradient agrees almost everywhere with this classical gradient. Since $|q|=1$ almost everywhere and $|X|=L^3<\infty$, it follows that $q\in W^{1,2}(X,\sph)$, with the same uniform energy bound. Values assigned on $S_E$ do not affect this Sobolev class.

We verify the inhomogeneous equation directly. Let $\Psi\in C^\infty(X,\R^3)$ be arbitrary. By the capacity-cutoff construction in Lemma~\ref{lem:capacity}, choose $\zeta_k\in C^\infty(X)$ such that
\[
0\le\zeta_k\le1,\qquad
\zeta_k=0\ \text{near }S_E,\qquad
\zeta_k\longrightarrow1\ \text{a.e.},\qquad
\norm{D\zeta_k}_{L^2(X)}\longrightarrow0.
\]
On $X\setminus S_E$ the equation $\Delta q+|Dq|^2q=e$ holds classically. Testing it with $\zeta_k\Psi$, whose support avoids $S_E$, and integrating by parts yields
\[
\begin{aligned}
&\int_X\zeta_k \langle Dq,D\Psi\rangle\dd x
+\int_X\sum_{i=1}^3(\partial_i\zeta_k)
\partial_iq\cdot\Psi\dd x\\
&\hspace{35mm}=\int_X\zeta_k(|Dq|^2q-e)\cdot\Psi\dd x.
\end{aligned}
\]
Here $\langle Dq,D\Psi\rangle=\sum_{i=1}^3\partial_iq\cdot\partial_i\Psi$. The additional cutoff term satisfies
\[
\left|\int_X\sum_{i=1}^3(\partial_i\zeta_k)
\partial_iq\cdot\Psi\dd x\right|
\le\norm\Psi_{L^\infty(X)}\norm{Dq}_{L^2(X)}
\norm{D\zeta_k}_{L^2(X)}\longrightarrow0.
\]
The other three integrands are dominated, respectively, by
\[
\norm{D\Psi}_{L^\infty(X)}|Dq|,\qquad
\norm\Psi_{L^\infty(X)}|Dq|^2,\qquad
\norm\Psi_{L^\infty(X)}|e|.
\]
All are integrable: $Dq\in L^2(X)$, $X$ has finite volume, and Lemma~\ref{lem:annulus} gives $e\in C^\infty(X,\R^3)$. Dominated convergence therefore gives
\[
\int_X\langle Dq,D\Psi\rangle\dd x
=\int_X(|Dq|^2q-e)\cdot\Psi\dd x.
\]
This is precisely \eqref{eq:weak-tension}. In particular, the smooth field $e$ is the actual distributional tension of the Sobolev map, with no additional distribution supported on $S_E$.

\emph{Essential singularities.}
Smoothness on $X\setminus S_E$ already gives $\sing(q)\subset S_E$. Fix $a=(0,0,z_0)\in S_E$, and choose $s>0$ small enough that $B_s(a)$ lies in a flat coordinate neighborhood. Because a null closed subset of the circle contains no open arc, $\T\setminus E$ is dense. We may therefore choose $z_+\notin E$ with $|z_+-z_0|<s/4$, using a local lift of the circle coordinate. The points
\[
a_+=(0,0,z_+),\qquad a_-=(s/4,0,z_0)
\]
belong to $B_s(a)\setminus S_E$. The construction gives $q(a_+)=\Np$ and $q(a_-)=\Sp$.

For every $0<\gamma<1$, continuity at these two regular points gives open balls $V_+$ and $V_-$, with closures contained in $B_s(a)\setminus S_E$, such that
\[
|q-\Np|<\gamma\quad\text{on }V_+,\qquad
|q-\Sp|<\gamma\quad\text{on }V_-.
\]
Both balls have positive three-dimensional measure. Since $|\Np-\Sp|=2$, the triangle inequality gives
\[
|q(x_+)-q(x_-)|\ge2-2\gamma
\qquad(x_+\in V_+,\ x_-\in V_-).
\]
Their Cartesian product has positive product measure, so this lower bound implies
\[
\operatorname*{ess\,sup}_{(x,y)\in B_s(a)\times B_s(a)}
|q(x)-q(y)|\ge2-2\gamma.
\]
Letting $\gamma\downarrow0$ and using $\operatorname{diam}\sph=2$, we obtain
\[
\operatorname*{ess\,sup}_{x,y\in B_s(a)}|q(x)-q(y)|=2
\quad\text{for every sufficiently small }s>0.
\]
This identity is invariant under changes on null sets. A representative $\widehat q$ continuous at $a$ would satisfy
\[
\operatorname{diam}\widehat q(B_s(a))
\le2\sup_{x\in B_s(a)}|\widehat q(x)-\widehat q(a)|\to0,
\]
a contradiction. Thus every point of $S_E$ is an essential discontinuity and $\sing(q)=S_E$.
\end{proof}

\subsection{The global stress identity}\label{sec:stress}
We now establish the global distributional hypothesis of Theorem~\ref{prop:repair}. Define the local reference map, for $r>0$,
\begin{equation}\label{eq:reference-map}
Q(r,\theta,z)
=\frac{(2\la(z)r\cos\theta,2\la(z)r\sin\theta,\la(z)^2-r^2)}
{r^2+\la(z)^2}.
\end{equation}
It is smooth in $z$ for fixed $r>0$, including $z\in E$, where it equals $\Sp$. The reference map is used only near the axis and need not be extended periodically in the transverse variables.

\begin{lemma}\label{lem:stress-comparison}
For a fixed sufficiently small $\delta$, the tensor difference on a tube of radius $r$ satisfies
\begin{equation}\label{eq:stress-L1}
\int_{\T}\int_0^{2\pi}r|T_q-T_Q|\dd\theta\dd z\longrightarrow0
\qquad(r\downarrow0).
\end{equation}
\end{lemma}
\begin{proof}
Use $Z$ and $\eta_r$ from the core estimates in the proof of Proposition~\ref{prop:q-properties}. On $r<\rho(z)/2$, the estimate \eqref{eq:phi-bounds} gives
\begin{equation}\label{eq:gradient-difference}
|DQ|\le CZ/r,\qquad |Dq-DQ|\le C(Z/r)\eta_r.
\end{equation}
For clarity, differentiating the spherical frame contributes terms $|\phi||F_r|$ and $|\phi||F_z|$, in addition to $|\phi_r|$ and $|\phi_z|$. They obey \eqref{eq:gradient-difference} because $Z\le1$ and $r/h\ll1$. The angular term is bounded by $|\sin f-\sin F|/r\le|\phi|/r$.
It follows that
\begin{equation}\label{eq:stress-core}
r|T_q-T_Q|\le CrZ^2\ell_\la(r)/h^2
\le C\la\ell_\la(2\rho)/h^2.
\end{equation}
The last inequality uses $\sup_{r>0}rZ(r,z)^2\le C\la(z)$. The majorant
\[
b_1(z)=C\la(z)\ell_{\la(z)}(2\rho(z))/h(z)^2
\]
extends flatly by zero on $E$ and belongs to $L^1(\T)$. For every fixed $z\notin E$, $r|T_q-T_Q|\to0$ as $r\downarrow0$.

On $r\ge\rho(z)/2$, \eqref{eq:q-annulus} gives $|Dq|\le C\la\rho^{-2}$ wherever $q$ is nonconstant. From \eqref{eq:F-derivatives}, the reference map satisfies
\[
|DQ|\le C\la r^{-2}+C\la h^{-1}r^{-1}.
\]
Since $\rho\ll h$, for all sufficiently small tube radii this implies
\begin{equation}\label{eq:stress-outer}
r|T_q-T_Q|\le C\la^2\rho^{-3}.
\end{equation}
For the constant exterior only the reference term remains; the same bound follows from $r\ge\rho/2$. The majorant $b_2=C\la^2\rho^{-3}$ also extends flatly by zero and belongs to $L^1(\T)$. Moreover,
\[
1_{\{r\ge\rho(z)/2\}}\longrightarrow0\quad(z\notin E),
\qquad q(r,\theta,z)=Q(r,\theta,z)=\Sp\quad(z\in E,\ r>0).
\] Dominated convergence applied to \eqref{eq:stress-core} and \eqref{eq:stress-outer} proves \eqref{eq:stress-L1}.
\end{proof}

\begin{proposition}\label{prop:stress}
On all of $X$, in distributions,
\begin{equation}\label{eq:global-stress}
\Div T_q=\cF,\qquad \cF_j=-2e\cdot\partial_jq.
\end{equation}
The field $\cF$ is smooth, has zero mean, and satisfies
\begin{equation}\label{eq:force-Cm}
\norm\cF_{C^m}\le C_m\delta^2B^{m+5}.
\end{equation}
\end{proposition}
\begin{proof}
Smoothness and \eqref{eq:force-Cm} follow from Lemma~\ref{lem:annulus}. Identity \eqref{eq:stress-identity-smooth} holds off $S_E$. It suffices to show that no boundary term remains after excluding a shrinking tube around the whole axis.

%\Needspace{6\baselineskip}

In the orthonormal cylindrical frame the reference map has
\begin{align}
(T_Q)_{rr}&=F_z^2=\frac{4r^2\la'^2}{(r^2+\la^2)^2},
& (T_Q)_{r\theta}&=0,\label{eq:Q-stress-r}\\
r(T_Q)_{rz}&=\frac{8r^2\la\la'}{(r^2+\la^2)^2}
=-4\partial_z\left(\frac{r^2}{r^2+\la^2}\right).&&\label{eq:Q-stress-z}
\end{align}
Let $Y$ be a smooth test vector field and $e_r=(\cos\theta,\sin\theta)$. Taylor expansion gives
\[
\int_0^{2\pi}Y_r(r,\theta,z)\dd\theta
=Y_\perp(0,0,z)\cdot\int_0^{2\pi}e_r\dd\theta+O(r)=O(r).
\]
The radial flux is therefore bounded by an integrable multiple of $\la'^2$, since
\[
r^2(T_Q)_{rr}\le4\la'^2.
\]
It tends to zero pointwise in $z$; on $E$ the derivative $\la'$ vanishes.

For the axial flux use \eqref{eq:Q-stress-z} and integrate by parts in the periodic $z$ coordinate. The result, with the outward radial orientation, is
\begin{equation}\label{eq:axial-flux}
4\int_\T\int_0^{2\pi}
\frac{r^2}{r^2+\la(z)^2}
\partial_zY_z(r\cos\theta,r\sin\theta,z)\dd\theta\dd z.
\end{equation}
Since
\[
0\le\frac{r^2}{r^2+\la(z)^2}\le1,\qquad
\frac{r^2}{r^2+\la(z)^2}\longrightarrow1_E(z)=0\quad\text{a.e.},
\]
dominated convergence makes \eqref{eq:axial-flux} tend to zero. This is the step that uses the nullity of $E$.

Lemma~\ref{lem:stress-comparison} transfers the vanishing flux to $T_q$ for every test vector field. Integration by parts outside the tube, followed by $r\downarrow0$, is justified in the volume terms by $T_q\in L^1$, as follows from Proposition~\ref{prop:q-properties}, and smoothness of $\cF$. It gives
\[
\int_X \langle T_q,DY\rangle\dd x=-\int_X\cF\cdot Y\dd x
\quad\forall\,Y\in C^\infty(X,\R^3),
\]
which is \eqref{eq:global-stress}. Taking $Y$ constant yields $\int_X\cF=0$.
\end{proof}

\begin{proof}[Proof of Proposition~\ref{prop:verification}]
The nullity of $E$ gives $\cH^1(S_E)=0$, hence $\capacity_2(S_E)=0$ by the capacity criterion in Subsection~\ref{sec:capacity}. Lemma~\ref{lem:annulus} gives the smooth flat extensions and the residual estimates. Proposition~\ref{prop:q-properties} gives finite energy, the global tension equation, and the exact essential singular set. Proposition~\ref{prop:stress} gives the global stress identity and zero mean. Finally,
\[
\delta(1+|\log\delta|)^k\longrightarrow0\quad(\delta\downarrow0)
\qquad\text{for every fixed }k\ge0,
\]
which gives the required smallness.
\end{proof}

\section{Proof of the main results}\label{sec:main-proof}
In this section we prove Theorem~\ref{thm:main} and Corollary~\ref{cor:sharp}. Let $E\subset\T$ be a nonempty compact null set, and let $q_\delta,e_\delta,\cF_\delta$ be as in Proposition~\ref{prop:verification}. For sufficiently small $\delta>0$, define $V_\delta,H_\delta$ by \eqref{eq:poisson-H} and $G_\delta$ by \eqref{eq:metric}. We show that
\[
U_\delta=(q_\delta,\mathrm{id}):X\longrightarrow(\sph\times X,G_\delta)
\]
satisfies all the assertions of Theorem~\ref{thm:main}. A full-dimensional null set, followed by restriction, rescaling, and extension independent of the additional domain variables, gives Corollary~\ref{cor:sharp}. We then describe the weak tangents and the geometry of the target.

\subsection{The stationary graph on the three-torus}
\begin{proof}[Proof of Theorem~\ref{thm:main}]
Apply Theorem~\ref{prop:repair} to the map of Proposition~\ref{prop:construction}, using Proposition~\ref{prop:verification}. By \eqref{eq:residual-Cm} and \eqref{eq:force-Cm},
\[
\|e_\delta\|_{C^0}+\|\cF_\delta\|_{C^0}
\le C(\delta B^3+\delta^2B^5)\to0,
\qquad \delta B^k\to0\quad\text{for every fixed }k\ge0.
\] The resulting metric is the explicit metric \eqref{eq:metric}, and $U_\delta=(q_\delta,\mathrm{id})$ is weakly harmonic and stationary on all of $X$.

The identity factor is smooth and does not remove the essential discontinuities of the first component. Hence $\sing U_\delta=S_E$. For every fixed $m$, \eqref{eq:metric-Cm} gives
\begin{equation}\label{eq:metric-limit}
\norm{G_\delta-\Gprod}_{C^m}
\le C_m\bigl(\delta B^{m+3}+\delta^2B^{m+5}\bigr)\longrightarrow0.
\end{equation}
Thus the convergence is in $C^\infty$. Along the graph, $e\cdot q=0$, so
\[
|DU_\delta|_{G_\delta}^2=|Dq_\delta|^2+3+\tr H_\delta.
\]
The uniform energy bound follows from Proposition~\ref{prop:q-properties} and \eqref{eq:H-Cm}.
\end{proof}

\subsection{A full-dimensional null set and the dimensional extension}
We give a direct construction of the null set needed for the sharp dimension statement, so that the endpoint dimension does not rely on a qualitative description of a Cantor set.

\begin{lemma}\label{lem:dimension-one}
There is a compact set $E\subset[0,1]$ with Lebesgue measure zero and $\dim_HE=1$.
\end{lemma}
\begin{proof}
At level $k\ge0$ retain $2^k$ closed intervals of equal length
\[
\ell_k=\frac{2^{-k}}{k+1}.
\]
Each level-$k$ interval retains the two level-$(k+1)$ intervals at its endpoints. This is possible because $2\ell_{k+1}<\ell_k$. Writing the level-$k$ intervals as $I_{k,j}$, we have
\[
E=\bigcap_{k\ge0}\bigcup_{j=1}^{2^k}I_{k,j},\qquad
\cL^1(E)\le2^k\ell_k=\frac1{k+1}\longrightarrow0.
\]
The nested intersection is compact.

Give each level-$k$ interval mass $2^{-k}$; consistency defines a probability measure $\mu$ on $E$. If $\ell_k\le r<\ell_{k-1}$, an interval of length $r$ meets at most six level-$k$ intervals: these intervals are disjoint and all have length $\ell_k$, and $\ell_{k-1}/\ell_k\le4$. Hence, for every $0<s<1$,
\[
\mu(I)\le6\,2^{-k}\le C_s\ell_k^s\le C_s r^s.
\]
Indeed $2^{-k}/\ell_k^s=2^{-k(1-s)}(k+1)^s$ is bounded. For any interval cover of $E$, the last estimate implies $1\le C_s\sum |I|^s$, so $\dim_HE\ge s$. Letting $s\uparrow1$ proves the assertion.
\end{proof}

\begin{proof}[Proof of Corollary~\ref{cor:sharp}]
Scale and translate the set of Lemma~\ref{lem:dimension-one} into a small coordinate arc of $\T$. A Euclidean coordinate ball in $X$ can be chosen to contain its whole axial copy. Restrict the map from Theorem~\ref{thm:main} to this ball and rescale the domain to obtain the three-dimensional assertion for any desired radius.

For $n>3$, put $d=n-3$ and define
$u(x',x'')=U(x')$ for $(x',x'')\in B_R^n\subset\R^3\times\R^d$, where $U$ is the rescaled map on $B_R^3$. In a fixed smooth embedding of $N$, Fubini's theorem identifies the weak derivatives as
$D_{x'}u=D U(x')$ and $D_{x''}u=0$. If $\omega_d=\cL^d(B_1^d)$, then
\[
\begin{aligned}
E_G(u;B_R^n)
&=\omega_d\int_{B_R^3}(R^2-|x'|^2)^{d/2}|DU(x')|_G^2\dd x'\\
&\le\omega_dR^d E_G(U;B_R^3)<\infty.
\end{aligned}
\]
Compactness of $N$ also gives the required $L^2$ bound, so $u\in W^{1,2}(B_R^n,N)$. For any smooth target test field $\Xi(x',x'',p)$ with compact domain support in $B_R^n$, fix $x''$ and apply the weak equation of $U$ to the slice $\Xi(\cdot,x'',p)$. Extend this field by zero outside its domain support. Integrating the resulting identity in $x''$ is justified by the finite-energy bound above. Since $D_{x''}u=0$, this is the full weak target-variation equation for $u$, including test fields that depend on $x''$.

The stress tensor satisfies
\[
T_u=\begin{pmatrix}
T_U&0\\0&|DU|_G^2I_{n-3}
\end{pmatrix},\qquad
\Div_{\R^n}T_u=
\begin{pmatrix}\Div_{\R^3}T_U\\D_{x''}|DU|_G^2\end{pmatrix}=0.
\]
These are distributional identities. More explicitly, extend a test field $Y=(Y',Y'')\in C_c^\infty(B_R^n,\R^n)$ by zero and use Fubini to write
\[
\begin{aligned}
\int_{B_R^n}\langle T_u,DY\rangle
&=\int_{\R^d}\int_{B_R^3}\langle T_U,D_{x'}Y'\rangle\dd x'\dd x''\\
&\quad+\int_{B_R^3}|DU|_G^2
\left(\int_{\R^d}\Div_{x''}Y''\dd x''\right)\dd x'=0.
\end{aligned}
\]
The first term vanishes by stationarity of $U$, and the second by compact support in $x''$. Thus the extension is stationary under arbitrary domain variations.

Denote the axial set in the rescaled three-ball again by $S_E$. The extended singular set is $(S_E\times\R^{n-3})\cap B_R^n$. This set has dimension $n-2$. For the lower bound, choose a small $(n-3)$-dimensional cube $Q$ such that the entire product $S_E\times Q$ lies inside $B_R^n$; this is possible because $S_E$ is compactly contained in the three-ball. Take the product of the measures in the proof of Lemma~\ref{lem:dimension-one} with Lebesgue measure on $Q$. For every $0<s<1$, the product assigns at most $Cr^{s+n-3}$ to a ball of radius $r$, and therefore has support of dimension at least $s+n-3$. The upper bound follows from containment in an $(n-2)$-plane. Within that plane the set has Lebesgue measure zero by Fubini's theorem, so its $\cH^{n-2}$ measure is zero. The arbitrary $C^\infty$ closeness of the target metric follows from \eqref{eq:metric-limit}.
\end{proof}

% Embedded references; no external bibliography database is required.


\begin{thebibliography}{99}

\bibitem{Bethuel1993} F. Bethuel, \emph{On the singular set of stationary harmonic maps}, Manuscripta Math. \textbf{78} (1993), 417--443.
\bibitem{BCL1986} H. Brezis, J.-M. Coron, and E. H. Lieb, \emph{Harmonic maps with defects}, Comm. Math. Phys. \textbf{107} (1986), 649--705.
\bibitem{ChangWangYang1999} S.-Y. A. Chang, L. Wang, and P. C. Yang, \emph{Regularity of harmonic maps}, Comm. Pure Appl. Math. \textbf{52} (1999), 1099--1111.
\bibitem{CheegerNaber2013} J. Cheeger and A. Naber, \emph{Quantitative stratification and the regularity of harmonic maps and minimal currents}, Comm. Pure Appl. Math. \textbf{66} (2013), 965--990.
\bibitem{CJN2026} B. Chow, W. Jiang, and A. Naber, \emph{Singular Sets of Harmonic Maps and Quantitative Stratification}, preprint, 2026.
\bibitem{DingLiLi2003} W. Ding, J. Li, and W. Li, \emph{Nonstationary weak limit of a stationary harmonic map sequence}, Comm. Pure Appl. Math. \textbf{56} (2003), 270--277.
\bibitem{EellsLemaire1978} J. Eells and L. Lemaire, \emph{A report on harmonic maps}, Bull. London Math. Soc. \textbf{10} (1978), 1--68.
\bibitem{EellsLemaire1988} J. Eells and L. Lemaire, \emph{Another report on harmonic maps}, Bull. London Math. Soc. \textbf{20} (1988), 385--524.
\bibitem{EellsSampson1964} J. Eells, Jr. and J. H. Sampson, \emph{Harmonic mappings of Riemannian manifolds}, Amer. J. Math. \textbf{86} (1964), 109--160.
\bibitem{ElSoufi1995} A. El Soufi, \emph{Indice de Morse des applications harmoniques de la sph\`ere}, Compositio Math. \textbf{95} (1995), 343--362.
\bibitem{Evans1991} L. C. Evans, \emph{Partial regularity for stationary harmonic maps into spheres}, Arch. Rational Mech. Anal. \textbf{116} (1991), 101--113.
\bibitem{HardtLin1987} R. Hardt and F.-H. Lin, \emph{Mappings minimizing the $L^p$ norm of the gradient}, Comm. Pure Appl. Math. \textbf{40} (1987), 555--588.
\bibitem{HardtLin1990} R. Hardt and F.-H. Lin, \emph{The singular set of an energy minimizing map from $B^4$ to $S^2$}, Manuscripta Math. \textbf{69} (1990), 275--289.
\bibitem{HLP1992} R. Hardt, F.-H. Lin, and C.-C. Poon, \emph{Axially symmetric harmonic maps minimizing a relaxed energy}, Comm. Pure Appl. Math. \textbf{45} (1992), 417--459.
\bibitem{Helein1990} F. H\'elein, \emph{R\'egularit\'e des applications faiblement harmoniques entre une surface et une sph\`ere}, C. R. Acad. Sci. Paris S\'er. I Math. \textbf{311} (1990), 519--524.
\bibitem{Helein1991} F. H\'elein, \emph{R\'egularit\'e des applications faiblement harmoniques entre une surface et une vari\'et\'e riemannienne}, C. R. Acad. Sci. Paris S\'er. I Math. \textbf{312} (1991), 591--596.
\bibitem{Helein2002} F. H\'elein, \emph{Harmonic Maps, Conservation Laws and Moving Frames}, 2nd ed., Cambridge Tracts in Mathematics 150, Cambridge University Press, 2002.
\bibitem{Hong1999} M.-C. Hong, \emph{On the Hausdorff dimension of the singular set of stable-stationary harmonic maps}, Comm. Partial Differential Equations \textbf{24} (1999), 1967--1985.
\bibitem{HongWang1999} M.-C. Hong and C.-Y. Wang, \emph{On the singular set of stable-stationary harmonic maps}, Calc. Var. Partial Differential Equations \textbf{9} (1999), 141--156.
\bibitem{Hsu2005} D. Hsu, \emph{An approach to the regularity for stable-stationary harmonic maps}, Proc. Amer. Math. Soc. \textbf{133} (2005), 2805--2812.
\bibitem{HsuLi2008} D. L. Hsu and J. Li, \emph{On the regularity for stationary harmonic maps}, Acta Math. Sin. (Engl. Ser.) \textbf{24} (2008), 223--226.
\bibitem{JagerKaul1983} W. J\"ager and H. Kaul, \emph{Rotationally symmetric harmonic maps from a ball into a sphere and the regularity problem for weak solutions of elliptic systems}, J. Reine Angew. Math. \textbf{343} (1983), 146--161.
\bibitem{Jost1991} J. Jost, \emph{Two-Dimensional Geometric Variational Problems}, Wiley, Chichester, 1991.
\bibitem{Karpukhin2021} M. Karpukhin, \emph{Index of minimal spheres and isoperimetric eigenvalue inequalities}, Invent. Math. \textbf{223} (2021), 335--377.
\bibitem{KarpukhinStern2024} M. Karpukhin and D. L. Stern, \emph{Existence of harmonic maps and eigenvalue optimization in higher dimensions}, Invent. Math. \textbf{236} (2024), 713--778.

\bibitem{Kra} J. Krantz, \emph{Partial Regularity of Stable Stationary Harmonic Maps into Certain Lie Groups},  arXiv:2605.03809v1 [math.DG] , \textbf{2026}.
\bibitem{LiTian1998} J. Li and G. Tian, \emph{A blow-up formula for stationary harmonic maps}, Internat. Math. Res. Notices \textbf{1998}, no.~14, 735--755.
\bibitem{LiX}X. Li, \emph{ Optimal regularity of stable harmonic maps to spheres}, arXiv:2608.20272v2 [math.AP], \textbf{2026}.
\bibitem{Lin1999} F.-H. Lin, \emph{Gradient estimates and blow-up analysis for stationary harmonic maps}, Ann. of Math. (2) \textbf{149} (1999), 785--829.
\bibitem{LinRiviere2002} F.-H. Lin and T. Rivi\`ere, \emph{Energy quantization for harmonic maps}, Duke Math. J. \textbf{111} (2002), 177--193.
\bibitem{LinWang2006} F.-H. Lin and C.-Y. Wang, \emph{Stable stationary harmonic maps to spheres}, Acta Math. Sin. (Engl. Ser.) \textbf{22} (2006), 319--330.
\bibitem{LinWang2008} F.-H. Lin and C.-Y. Wang, \emph{The Analysis of Harmonic Maps and Their Heat Flows}, World Scientific, 2008.
\bibitem{Luckhaus1988} S. Luckhaus, \emph{Partial H\"older continuity for minima of certain energies among maps into a Riemannian manifold}, Indiana Univ. Math. J. \textbf{37} (1988), 349--367.
\bibitem{NV2017} A. Naber and D. Valtorta, \emph{Rectifiable-Reifenberg and the regularity of stationary and minimizing harmonic maps}, Ann. of Math. (2) \textbf{185} (2017), 131--227.
\bibitem{NVEnergy2024} A. Naber and D. Valtorta, \emph{Energy identity for stationary harmonic maps}, arXiv:2401.02242, 2024; revised 2026.
\bibitem{Nakajima2006} T. Nakajima, \emph{Singular points of harmonic maps from $4$-dimensional domains into $3$-spheres}, Duke Math. J. \textbf{132} (2006), 531--543.
\bibitem{Nakajima2009} T. Nakajima, \emph{A remark on instability of harmonic maps between spheres}, Pacific J. Math. \textbf{240} (2009), 363--369.
\bibitem{Okayasu1994} T. Okayasu, \emph{Regularity of minimizing harmonic maps into $S^4$, $S^5$ and symmetric spaces}, Math. Ann. \textbf{298} (1994), 193--205.
\bibitem{Riviere1995} T. Rivi\`ere, \emph{Everywhere discontinuous harmonic maps into spheres}, Acta Math. \textbf{175} (1995), 197--226.
\bibitem{RiviereStruwe2008} T. Rivi\`ere and M. Struwe, \emph{Partial regularity for harmonic maps and related problems}, Comm. Pure Appl. Math. \textbf{61} (2008), 451--463.
\bibitem{SacksUhlenbeck1981} J. Sacks and K. Uhlenbeck, \emph{The existence of minimal immersions of $2$-spheres}, Ann. of Math. (2) \textbf{113} (1981), 1--24.
\bibitem{SU1982} R. Schoen and K. Uhlenbeck, \emph{A regularity theory for harmonic maps}, J. Differential Geom. \textbf{17} (1982), 307--335.
\bibitem{Simon1995} L. Simon, \emph{Rectifiability of the singular set of energy minimizing maps}, Calc. Var. Partial Differential Equations \textbf{3} (1995), 1--65.
\bibitem{Simon1996} L. Simon, \emph{Theorems on Regularity and Singularity of Energy Minimizing Maps}, Lectures in Mathematics ETH Z\"urich, Birkh\"auser, 1996.
\bibitem{Xin1980} Y. L. Xin, \emph{Some results on stable harmonic maps}, Duke Math. J. \textbf{47} (1980), 609--613.

\end{thebibliography}
\end{document}